\documentclass[reqno]{amsart}
\usepackage{DocumentPreamble}

\AddTitle{On the satisfaction frequency of spectral characterization conditions}
\AddAuthor{Nikita Lvov\textsuperscript{$\diamond$}}
\address{$\diamond$ McGill University, Canada}
\email{\href{mailto:nikita.lvov@mail.mcgill.ca}{nikita.lvov@mail.mcgill.ca}}
\AddAuthor{Alexander Van Werde\textsuperscript{$\ast$}}
\address{$\ast$ University of Münster, Germany}
\email{\href{mailto:a.van.werde@uni-muenster.de}{a.van.werde@uni-muenster.de}}

\AddKeyword{cospectral}
\AddKeyword{cokernel}
\AddKeyword{walk matrix}
\AddKeyword{discriminant}
\AddKeyword{symmetric integer matrix} 
\AddKeyword{profinite model}

\AddMSC{05C50}
\AddMSC{05C80}
\AddMSC{60B20}
\AddMSC{15B52}

\begin{document}
\begin{abstract}    
    We give the first specific conjectures on how frequently graphs satisfy sufficient conditions for being uniquely characterized by spectral information. 
    These conjectures arise from a theoretical framework that we developed based on abstract-algebraic random matrix statistics.  
    Specifically, we rephrase conditions from the literature in terms of $\bbZ[x]$-modules associated to the adjacency matrix, and study the distribution of those modules in analytically tractable profinite random matrix ensembles. 
    We applied this new method to two distinct conditions. 
    The first requires square-freeness of the determinant of the walk matrix, and the second uses the discriminant of the characteristic polynomial. 
\end{abstract}
\maketitle

\section{Introduction}
It has been known since the 1950s that there exist non-isomorphic graphs whose adjacency matrices have the same characteristic polynomial, called \emph{cospectral mates} \cite{von1957spektren}. 
A fundamental conjecture by Haemers however proposes that such examples should be atypical \cite{haemers2004enumeration,van2003graphs}. 
That is, if one picks a graph on $n$ nodes uniformly at random, then it is believed that the probability of cospectrality tends to zero as $n\to \infty$.
A closely related \emph{fingerprint conjecture} for matrices with entries $\pm 1$ was more recently also proposed by Vu \cite{vu2021recent,VU2014combinatorial}.

Haemers' conjecture remains open. 
The current best lower bound on the number of graphs characterized by spectrum is by Koval and Kwan \cite{koval2024exponentially} who established that at least $\exp(cn)$ graphs on $n$ nodes are characterized by adjacency spectrum for some $c>0$. 
This was a significant improvement on previous bounds of order $\exp(c\sqrt{n})$, but remains a vanishing fraction of all $(1-o(1))2^{n(n-1)/2}/n!\geq \exp(cn^2)$ graphs.
Moreover, it seems likely that an exponential bound is the limit of the constructive methods in \cite{koval2024exponentially} and that even bounds of order $\exp(n^{1+\varepsilon})$ would require  different methods: ``...It is hard to imagine a natural argument that could reconstruct so many different graphs... Instead, it seems that non-constructive methods may be necessary...'' \cite[\S 1.1]{koval2024exponentially}.  

There indeed exist sufficient conditions that can certify that a graph is determined by spectral information without giving an algorithm to reconstruct the graph from its spectrum. 
Such conditions were first discovered by Wang and Xu \cite{wang2006sufficient} for a generalized notion where one is also given the spectrum of the complement graph, and considerable effort has since been devoted to refinements of these sufficient conditions  \cite{wang2017simple,qiu2023smith,wang2021generalized,wang2025haemers,guo2025primary,wang2023improved,wang2016square,chao2025new,yang2024improved,van2025Sufficient}.
(We give examples in \Cref{sec: IntroPredictions}.)

Numerical investigations suggest that such sufficient conditions are applicable with non-vanishing frequency. 
This would be significant considering the expected limits of constructive methods.  
Unfortunately, the current theoretical understanding of the probabilistic behavior of the conditions is also limited.
For instance, while the satisfaction frequencies of the conditions can be roughly estimated numerically, it is not known for any of them what the exact limiting values might be. 

\pagebreak[4]
The present paper takes an initial step towards a probabilistic understanding of spectral characterization condition by developing a connection to the theory of abstract-algebraic random matrix statistics.  
Our approach has two main steps. 
The first is to rephrase the considered condition in terms of an abstract-algebraic object associated to the graph's adjacency matrix, such as certain groups or a modules.
The second step is to study the distribution of those algebraic objects using an analytically tractable profinite random matrix ensemble. 
The ensemble used in this second step incorporates the symmetry constraint on the adjacency matrix that comes from the the graph being undirected, while achieving analytical tractability by taking the entries uniform on a larger space than $\{0,1 \}$.  
In particular, the results achieved in the latter model give theoretical insight on the limiting satisfaction frequency of various conditions; see  \Cref{thm: HaarMainWalk,thm: MainHaarDisc}.    

Our main results require some technical background to motivate and state, which we build up over the span of the paper. 
A notable and easy-to-state consequence is that we find the first specific predictions on the satisfaction frequency of spectral characterization conditions.  
The resulting conjectures are showcased in \Cref{sec: IntroPredictions}. 
We explain the framework that we developed in \Cref{sec: IntroTwoStep} and comment on universality in \Cref{sec: IntroUniversality}.  
The remainder of the paper is outlined in \Cref{sec: IntroOutline}. 

\subsection{Specific predictions for satisfaction frequencies}\label{sec: IntroPredictions}
To demonstrate the flexibility of our approach, we consider two different conditions. 
The first is based on the square-freeness of the determinant of the walk matrix and is the subject of \Cref{conj: WalkMatrixSquareFree}. 
The second uses the discriminant of the characteristic polynomial and is considered in \Cref{conj: Discriminant}. 

\subsubsection{Conditions based on the walk matrix}\label{sec: IntroWalkMatrix}
The following definitions are due to Farrugia \cite{farrugia2019overgraphs} and Qiu, Ji, Mao, and Wang \cite{qiu2023generalized} in a special case; see also \cite{van2025Sufficient,wang2016square} for the general setting used here.

\begin{definition}\label{def: IsomorphicSignedSpectrum}
    Consider an integer vector $\zeta\in \bbZ^n$ and a matrix $\bM\in \bbZ^{n\times n}$ that is symmetric $\bM = \bM^{\T}$.
    Two such pairs $(\bM,\zeta)$ and $(\bN,\eta)$ are \emph{isomorphic up to signed permutation} if there exists a signed permutation $\bS$ such that $\bN = \bS \bM \bS^{\T}$ and $\bS \zeta = \zeta$.
\end{definition}
Here, a \emph{signed permutation} is a matrix of the form $\bS = \bP \bD$ with $\bP$ a permutation and $\bD$ a diagonal matrix with entries $\pm 1$. 
The signs serve to avoid trivial obstructions to spectral characterization for general integer matrices, but are irrelevant if one restricts to the adjacency matrices of a graphs as two nonnegative matrices are related by a signed permutation if and only if they are related by an unsigned permutation. 
Consequently, since permutations corresponds to graph isomorphism, the following notion yields a generalization of adjacency spectral characterization, with some additional information related to the vector $\zeta$:

\begin{definition}\label{def: CharacterizedPhi}
    The \emph{bivariate characteristic polynomial} of the pair $(\bM,\zeta)$ is defined as 
    \begin{align}
        \Phi_{\bM,\zeta}(\lambda, t) \de \det\bigl(\lambda \bI - \bM - t \zeta \zeta^{\T}\bigr),\label{eq:BraveActor}
    \end{align}
    where $\bI$ is the identity.
    The pair $(\bM,\zeta)$ is said to be \emph{characterized by $\Phi$-spectrum} if every $(\bN,\eta)$ with $\Phi_{\bM,\zeta} = \Phi_{\bN, \eta}$ is isomorphic up to signed permutation. 
\end{definition}
The case where $\zeta = \bb1_S$ is the indicator vector of some set $S \subseteq \{1,\ldots,n\}$ and $\bM$ is the adjacency matrix of a graph admits a direct graph-theoretic interpretation.
Then, the differentiating information in $\Phi_{\bM,\zeta}$ can be shown to be equivalent to the pair consisting of the spectrum of $\bM$ and the spectrum of the graph on $n+1$ nodes found by connecting the vertices in $S$ to a new node; see Farrugia  \cite{farrugia2019overgraphs}. 
The special case where $\zeta = (1,\ldots,1)^{\T}$ is the all-ones vector is also known to be equivalent to considering the spectra of the graph and its complement, due to Johnson and Newman \cite{johnson1980note}.

We consider a sufficient condition in \cite[Theorem 2.12]{van2025Sufficient} that is a variant on one by Wang \cite{wang2017simple}.
The \emph{walk matrix} of $(\bM,\zeta)$ is the square matrix $\bW \in \bbZ^{n\times n}$ with columns $\zeta, \bM \zeta, \ldots , \bM^{n-1}\zeta$.
The terminology refers to the fact that if $\zeta = \bb1_S$ is an indicator vector and $\bM$ is the adjacency matrix of a graph, then the $ij$th entry $\bW_{i,j}$ counts walks of length $j-1$ that start in vertex $i$ and end in $S$.
Recall that an integer $d$ is said to be \emph{square-free} if $p^2\nmid d$ for every prime $p$.
\begin{theorem}[{Van Werde \cite{van2025Sufficient}}]\label{thm: WalkSufficientConditionWithLoops}
    Consider $\zeta\in \bbZ^n$ and $\bM\in \bbZ^{n\times n}$ with $\bM = \bM^{\T}$ and suppose that $\det(\bW)$ is square-free. 
    Then, $(\bM,\zeta)$ is characterized by $\Phi$-spectrum up to signed permutation. 
\end{theorem}
We would like to understand how often this condition is applicable.
For instance, one natural setting would be to take $\bM$ as the adjacency matrix of a random graph and $\zeta = \bb1_{S}$ as the indicator of some randomly chosen set of vertices. 
Results by O'Rourke and Touri show that $\det(\bW)\neq 0$ with high probability in such settings \cite[Theorems 1.5 \& 3.7]{o2016conjecture}; see also \cite{meehan2019eigenvectors,luh2021eigenvectors} for extensions. 

Unfortunately, \cite{o2016conjecture,meehan2019eigenvectors,luh2021eigenvectors} do not enable predictions for square-freeness.    
The results achieved in \Cref{sec: ConditionWalk} in different random matrix ensembles lead us to the following prediction for the satisfaction frequency of the condition in \Cref{thm: WalkSufficientConditionWithLoops} for graphs with loops:

\begin{conjecture}\label{conj: WalkMatrixSquareFree}
    Suppose that $\bM$ is chosen uniformly at random from all symmetric $n\times n$ matrices with $\{0,1 \}$-valued entries and consider an independent random vector $\zeta$ that is uniform on $\{0,1 \}^n$.   
    Then, for every prime $p$,  
    \begin{align}
        \lim_{n\to \infty}\bbP\bigl(p^2 \nmid \det(\bW)\bigr)  = \Bigl(1 - \frac{1}{p^2} - \frac{1}{p^3} + \frac{1}{p^4}\Bigr) \prod_{k=1}^\infty \Bigl(1 - \frac{1}{p^{2k}}\Bigr) \label{eq: DetWalkPsquare}
    \end{align} 
    Further, the limiting probability of square-freeness factorizes over the constituent primes:
    \begin{align}
        \lim_{n\to \infty}\bbP\bigl(\det(\bW)\textnormal{ is square-free}\bigr) &= \prod_{p\textnormal{ prime}} \Bigl(1 - \frac{1}{p^2} - \frac{1}{p^3} + \frac{1}{p^4}\Bigr) \prod_{k=1}^\infty \Bigl(1 - \frac{1}{p^{2k}}\Bigr) \label{eq: DetWalkSquarefree} \\ 
        &= 0.2943\ldots\nonumber
    \end{align}
\end{conjecture}
One can compare \eqref{eq: DetWalkPsquare} with empirical probabilities in \Cref{tab: WalkRandVec} and observe a good match up to at least the third significant digit.
Let us note that such a good match does not apply for the simpler heuristic probability $1-1/p^2$ that would arise if one considers the probability that $p^2\nmid D$ for a random integer $D$ uniform on $\{1,\ldots, m \}$ for $m$ large.
In fact, even if one only disregards the term $1/p^4$ in \eqref{eq: DetWalkPsquare} then this would already show a discrepancy on the second significant digit since $(1-1/p^2 - 1/p^3)\prod_{k\geq 1} (1-1/p^{2k})$ is approximately $0.430$ and $0.747$ for $p=2$ and $p=3$, respectively.

\begin{table}[h!]
    \begin{center}
    \begin{tabular}{ |c|c|c|c|c|c|c|c|c| }
        \hline
        $p^2\nmid \det(\bW)$&$n=8$&$n =10$&$n=12$&$n=15$&$n = 25$&$n=50$&$n=100$& \Cref{conj: WalkMatrixSquareFree}\\
        \hline
        $p=2$&0.451&0.465&0.470&0.472&0.473&0.473&0.473&0.47336955677\ldots\\
        $p=3$&0.616&0.690&0.731&0.752&0.758&0.757&0.757&0.75752129361\ldots\\
        $p=5$&0.689&0.801&0.867&0.905&0.914&0.914&0.914&0.91393033780\ldots \\
        $p=7$&0.708&0.830&0.904&0.946&0.957&0.956&0.957&0.95674525798\ldots\\
        $p=11$&0.719&0.848&0.926&0.972&0.983&0.983&0.983&0.98279431682\ldots\\ 
        \hline
      \end{tabular}
    \end{center}
    \caption{Estimated probability that $p^2\nmid \det(\bW)$ in the setting of \Cref{conj: WalkMatrixSquareFree}. 
    These estimates used $10^6$ samples and hence have an uncertainty of $\pm 0.0005$ in the sense of standard deviation.}
    \label{tab: WalkRandVec}
\end{table}

To our knowledge, \eqref{eq: DetWalkSquarefree} is the first well-motivated conjecture for the limiting satisfaction frequency for any sufficient condition for spectral characterization.   
A conceptual forerunner is our previous work that studied the walk matrix in a setting without symmetry constraint; see Van Werde \cite{van2025cokernel}.
That setting is substantially simpler because removing the symmetry constraint allows the entries of $\bM$ to all be independent which enables different methods, such as direct analysis of the stochastic process $\zeta, \bM \zeta, \ldots, \bM^{n-1}\zeta$ in the columns of $\bW$. 
However, the predictions in that setting do not match observations in symmetric settings \cite[Section 4]{van2025cokernel} and there are no sufficient conditions for non-symmetric matrices, so that there were no implications for spectral characterization.
A key technical innovation in the current work is that we can incorporate a symmetry constraint.

Moreover, the results in \cite{van2025cokernel} were specific to the walk matrix. 
The perspective developed in the current paper is more flexible. 
For instance, we next illustrate in \Cref{sec: IntroDiscriminant} that we also gain insight on spectral characterization conditions that are initially unrelated to the walk matrix.

\subsubsection{Conditions based on the discriminant}\label{sec: IntroDiscriminant}
Note that symmetric matrices $\bM,\bN$ are cospectral if and only if there exists an orthogonal matrix $\bQ$ such that $\bN = \bQ \bM \bQ^{\T}$. 
In many examples of cospectral graphs, although certainly not all, it further holds that $\bQ$ can be taken to have rational entries.   
For instance, this holds for examples produced by \emph{Godsil--McKay switching} \cite{godsil1982constructing} which exhaustive enumeration has shown to be responsible for a large fraction of cospectral graphs on $n\leq 12$ vertices \cite{haemers2004enumeration}.  
This may motivate the following relaxation of cospectrality:

\begin{definition}\label{def: CospectralThroughRational}
    Two symmetric matrices $\bM,\bN$ are said to be \emph{cospectral through a rational matrix} if there exists a orthogonal matrix $\bQ$ with entries in $\bbQ$ such that $\bN = \bQ \bM \bQ^{\T}$.  
\end{definition}

This notion is closely related to \Cref{def: CharacterizedPhi}. 
Indeed, given a pair $(\bM,\zeta)$ with nonsingular walk matrix $\det(\bW) \neq 0$ it can be shown that $\Phi$-cospectrality to some other pair $(\bN,\eta)$ is equivalent to the existence of a rational orthogonal matrix $\bQ$ with $\bN = \bQ \bM\bQ^{\T}$ and $\eta = \bQ \zeta$; see \cite[Lemma 2.1]{van2025Sufficient}.  
Here, recall that O'Rourke and Touri proved that the walk matrix is indeed nonsingular with high probability  \cite{o2016conjecture}.  
\Cref{def: CospectralThroughRational} hence provides at least as good a proxy for Haemers' conjecture as $\Phi$-cospectrality since the constraint that $\bQ$ has to be compatible with the vectors is now removed.
In particular, a bound on the fraction of graphs that are cospectral through a rational matrix would imply a bound on the fraction that are $\Phi$-cospectral.

We consider a condition from a work of Wang and Yu \cite{wang2016square}. 
The \emph{discriminant} of a monic polynomial $\phi \in \bbZ[x]$ is the integer $\Delta_\phi \in \bbZ$ given by the resultant of $\phi$ and its derivative $\phi'$; see \Cref{def: Discriminant} in \Cref{sec: PreliminariesDiscriminant}. 
If $\lambda_1,\ldots,\lambda_n\in \bbC$ are the roots of $\phi$, then one has the following equivalent expression:
\begin{equation}
    \Delta_\phi = \prod_{i=1}^n  \prod_{j\neq i} (\lambda_i - \lambda_j). \label{eq: DiscriminantRootFormula}
\end{equation}
Let $\Delta_{\bM} \de \Delta_{\varphi_{\bM}}$ denote the discriminant of the characteristic polynomial $\varphi_{\bM}(\lambda) \de \det(\lambda \bI -\bM)$.

\begin{theorem}[Wang and Yu {\cite{wang2016square}}]\label{thm: DiscriminantCondition}
    Consider symmetric integer matrix $\bM$ and suppose that the discriminant $\Delta_{\bM}$ is odd and square free. 
    Then, a symmetric integer $\bN$ is cospectral to $\bM$ through a rational matrix if and only if there exists a signed permutation $\bS$ with $\bN = \bS \bM \bS^{\T}$.  
\end{theorem}

Related conditions for pairs of matrices are given in work of Bhargava, Gross, and Wang \cite[Proposition 35]{bhargava2017positive}, and extensions to algebraic number fields will appear in our forthcoming \cite{vanwerde2026prep}.  

Our results achieved using abstract-algebraic random matrix statistics in \Cref{sec: ResultsDiscriminant} yield the following prediction for the satisfaction frequency of the condition in \Cref{thm: DiscriminantCondition}: 
\begin{conjecture}\label{conj: Discriminant}
    Suppose that $\bM$ is chosen uniformly at random from all symmetric $n\times n$ matrices with $\{0,1 \}$-valued entries. 
    Then, for every prime $p$, 
    \begin{equation}
        \lim_{n\to \infty}\bbP\bigl(p \nmid \Delta_{\bM}\bigr) = \Bigl( 1 - \frac{1}{p}\Bigr) \prod_{k=1}^\infty\Bigl(1 - \frac{1}{p^{2k}}\Bigr). \label{eq: Discriminant}
    \end{equation}
    Further, for every odd prime $p$, 
    \begin{equation}
        \lim_{n\to \infty}\bbP\bigl(p^2 \nmid \Delta_{\bM}\bigr) = \Bigl(1 - \frac{1}{p^2} \frac{3p-1}{p+1}\Bigr)\prod_{k=1}^\infty\Bigl(1 - \frac{1}{p^{2k}}\Bigr) \label{eq:BreezyGem}
    \end{equation}
    Moreover, the limiting probability for being odd and square-free factorizes over the constituent primes:
    \begin{align}
        \lim_{n\to \infty} \bbP\bigl(\Delta_{\bM} \textnormal{ is odd and square-free}\bigr) &= \frac{6}{7} \prod_{p \textnormal{ prime}}\Bigl(1 - \frac{1}{p^2} \frac{3p-1}{p+1}\Bigr)\prod_{k=1}^\infty\Bigl(1 - \frac{1}{p^{2k}}\Bigr) \label{eq:AncientGym}\\ 
        &= 0.1686\ldots \nonumber 
    \end{align}
\end{conjecture}
\Cref{conj: Discriminant} matches empirical evidence presented in \Cref{tab: Disc}. 
For comparison, let us note that a random monic polynomial $\phi \in \bbZ[x]$ with independent and uniform random coefficients is known to satisfy that $p^2 \nmid \Delta_{\phi}$ for odd $p$ with probability $ 1 - p^{-1} + (p-1)^2p^{-2}(p+1)^{-1} $ as $n\to \infty$; see \eg \cite[Section 6]{ash2007equality}. 
That heuristic would not match the evidence in \Cref{tab: Disc}. 
For instance, it would predict that $\bbP(3^2 \nmid \Delta_{\phi}) \approx 0.7777\ldots$ and $\bbP(5^2 \nmid \Delta_{\phi}) \approx 0.9066\ldots$. 

\begin{table}[h!]
    \begin{center}
    \begin{tabular}{ |c|c|c|c|c|c|c|c|c| }
        \hline
        $q\nmid \Delta_{\bM}$&$n=8$&$n =10$&$n=12$&$n=15$&$n = 25$&$n=50$&$n=100$& \Cref{conj: Discriminant}\\
        \hline
        $q = 2$&0.346&0.345&0.345&0.344&0.344&0.345&0.344&0.34426876856... \\
        $q=3^2$&0.628&0.662&0.672&0.679&0.682&0.682&0.682&0.68176916425...\\
        $q=5^2$&0.780&0.821&0.850&0.865&0.868&0.869&0.869&0.86894942632...\\
        $q=7^2$&0.843&0.885&0.909&0.925&0.929&0.929&0.929&0.92921742127...\\
        $q=11^2$&0.885&0.929&0.952&0.966&0.970&0.970&0.970&0.96981231057...\\ 
        \hline
      \end{tabular}
    \end{center}
    \caption{Estimated probability that $q\nmid  \Delta_{\bM}$ in the setting of \Cref{conj: Discriminant} for those values of $q$ that are relevant for the discriminant being odd and square-free. 
    These estimates used $10^6$ samples.}
    \label{tab: Disc}
\end{table}

\subsection{Two-step approach based on abstract-algebraic random matrix statistics}\label{sec: IntroTwoStep}
Let us now explain the theoretical framework that we developed and that resulted in the predicted values in Conjectures \ref{conj: WalkMatrixSquareFree} to \ref{conj: Discriminant}. 
The basic idea is given by the following two-step approach:
\begin{enumerate}[itemsep=0.5em]
    \item 
    A-priori it may appear that the walk matrix and the discriminant are totally different objects whose study would require different methods, but it turns out that they can both be expressed in terms of the structure of the same abstract-algebraic object.  
    Specifically, an integer matrix $\bM \in \bbZ^{n\times n}$  induces a $\bbZ[x]$-module structure on the $n$-dimensional lattice $\bbZ^n$ specified by $xv = \bM v$ for every $v\in \bbZ^n$. 
    As it turns out, both spectral characterization conditions can be rephrased in terms of the structure of this $\bbZ[x]$-module. 
    This is done in \Cref{sec: Rephrasing}. 
    \item
    The question becomes to understand the distribution of the aforementioned module. 
    To this end, we introduce an analytically tractable random matrix ensemble that incorporates key structure from $\bM$ and also induces a $\bbZ[x]$-module. 
    Our model here is the Haar distribution on symmetric matrices with entries in the profinite completion of $\bbZ[x]$. 
    \Cref{sec: ComputationAnalytical} defines this ensemble in detail and gives rigorous results on frequency with which the associated modules satisfy the constraints arising from spectral characterization conditions.

    The reason why this profinite ensemble is analytically tractable is that it enjoys a strong invariance property: if $\bM$ is a matrix from that ensemble and $\bG$ is a fixed invertible matrix over the profinite completion of $\bbZ[x]$, then $\bG\bM\bG^{\T}$ is again symmetric with the same distribution as $\bM$.  
    We crucially exploit this invariance the proof of \Cref{prop: Reduction} to establish a reduction from module-theoretic events associated to the $n\times n$ matrix $\bM$ to events stated using matrices of bounded size. 
    Given that reduction, all relevant probabilities can be deduced from a finite computation. 
\end{enumerate}

We believe that the two-step method itself has significant potential beyond the specific results that we establish in the present paper. 
There are many different sufficient conditions in the literature whose satisfaction frequency remains to be theoretically understood \cite{wang2017simple,qiu2023smith,wang2021generalized,wang2025haemers,guo2025primary,wang2023improved,wang2016square,chao2025new,yang2024improved,van2025Sufficient}.
Simplified analytically tractable models are attractive in this context, as they should enable theoretical insight for many different conditions at a relatively low cost.  
Of course, some adjustments may be necessary.
The appropriate category of abstract-algebraic objects considered in the first step may depend on the condition under consideration, and the selection of analytical tractable model  in the second step also requires some care. 
One has to incorporate enough structure of the original random matrix ensemble that one hopes to gain insight on, but the computation will become more complicated as more structure is retained.

The first step in our approach to rephrase using abstract-algebraic objects arguably has roots dating back all the way to the proofs of the original sufficient conditions of Wang and Xu \cite{wang2006sufficient}. 
Those made use of the Smith normal form of the walk matrix in \cite[Section 3]{wang2006sufficient}, a perspective which has recently also been emphasized in work on refinements of the conditions by Qiu, Wang, and Zhang \cite{qiu2023generalized}. 
The Smith normal form of an integer matrix is equivalent to the Abelian group structure of its cokernel. 
However, \cite{qiu2023generalized,wang2006sufficient} mostly seem to have viewed the Smith normal form as a sequence of integers that provides more detailed information than the determinant. 
That the viewpoint of abstract-algebraic objects and $\bbZ[x]$-module structure is useful in probabilistic results stems from our previous work \cite{van2025cokernel} with walk matrices for directed graphs; recall the discussion after \Cref{conj: WalkMatrixSquareFree}.

The second step of the approach is similar in spirit to the role of Gaussian ensembles in classical random matrix theory that are analytically tractable and enable detailed asymptotics for joint eigenvalue statistics, which can often be rigorously extended to other ensembles like random graphs by a universality argument; see \eg \cite{erdHos2013spectral,anderson2010introduction}.
We are however concerned with arithmetical statistics such as square-freeness and random abstract-algebraic objects. 
Profinite ensembles with Haar distributed entries are the natural analogue for Gaussian ensembles in this arithmetical setting. 

A related story can be found in the literature on \emph{sandpile groups} of random graphs, which are the cokernels of Laplacian matrices. 
Motivated by experimental findings by Clancy, Leake, and Payne \cite{clancy2015note}, those authors together with Kaplan and Wood \cite{clancy2015cohen} established rigorous results in analytically tractable $p$-adic random matrix models. 
The predictions following from those models were later shown to be correct for random graphs by universality arguments; see \eg Wood \cite{wood2017distribution}, Nguyen and Wood \cite{nguyen2025local}, and Hodges \cite{hodges2024distribution}. 
In particular, the results in \cite{clancy2015cohen,wood2017distribution} showed that the symmetry constraint on the Laplacian that comes from the graph being undirected has an important effect on the distribution.  
As we discuss in \Cref{sec: Background}, the profinite completion of $\bbZ$ is closely related to the rings of $p$-adic integers. 
Hence, \cite{clancy2015cohen} can also be interpreted as a profinite symmetric model.

For random $\bbZ[x]$-modules associated to random matrices, all previous works that we are aware of have focused on settings without symmetry constraint. 
In particular, the module structure of the cokernel of $P(\bM)$ when $\bM$ is a random matrix with independent entries and $P\in \bbZ[x]$ is a polynomial has been considered by various authors; see \eg \cite{cheong2023distribution,lee2023joint,cheong2022generalizations,cheong2025cokernel,cheong2023polynomial,shen2026eigenvalues,van2025cokernel}.
A \emph{linearization trick} is often used in that setting to reduce the study of the cokernel of the integer matrix $P(\bM) \in \bbZ^{n\times n}$ to that of the polynomial matrix $\bM - x\bI \in \bbZ[x]^{n\times n}$. 
We also use that trick in our our arguments; see  \Cref{sec: PrelimCokernel} as well as the discussion surrounding \eqref{eq:DampQuote} and \Cref{rem: CokernelDisc}.

The idea of our reduction approach exploiting invariance was sparked by proofs of Evans concerning Markovian structure in the Smith normal form of a random $p$-adic matrix without symmetries \cite{evans2002elementary}.  
After using the method in the current paper, we have also found application of such invariance in cospectrality problems that are not directly related to sufficient conditions, such as the the probability that $\bQ^{\T} \bM \bQ\in \bbZ^{n\times n}$ for a fixed rational matrix $\bQ$ when $\bM$ is a random symmetric integer matrix with sufficiently uniform entries; see Van Werde \cite{van2026exact}.

\subsection{On universality}\label{sec: IntroUniversality}
We believe that the considered statistics should not be too sensitive to the specific distribution of entries, so long as the key structure of the matrix is retained. 
Indeed, this belief is implicit in our approach, as it is what justifies passing to a different random matrix model.

Specifically, we expect that the limiting probabilities \eqref{eq: DetWalkPsquare}, \eqref{eq: Discriminant}, and \eqref{eq:BreezyGem} will be universal so long as the entries of the random matrix $\bM$ are not too concentrated on any residue class, in the sense of  balanced distributions \cite[Definition 1]{wood2019random}. 
Such universality is known to hold in related models; see Van Werde \cite[Theorem 1.3]{van2025cokernel} for cokernels of walk matrices and Cheong and Yu for $\bbZ[x]$-module statistics \cite[Theorem 1.3]{cheong2023distribution}, both without symmetry constraint, and see Wood \cite{wood2017distribution} for sandpile groups in a setting with symmetry. 
A sufficiently strong universality statement would make our conjectures into theorems, as our results for profinite models then identify the limiting law.

We intend to pursue such universality results in future work, but note that significant technical challenges remain to be resolved before a full proof of \eqref{eq: DetWalkSquarefree} and \eqref{eq:AncientGym} could be expected.
For instance, square-freeness depends on infinitely many primes simultaneously, which will likely necessitate separate methods to rule out large primes as certain universality techniques perform best when one only considers a finitely many primes simultaneously; see \eg \cite[Section 1.5]{nguyen2025local} for discussion of such issues in the context of sandpile groups.  
In fact, even \eqref{eq: DetWalkPsquare}, \eqref{eq: Discriminant}, and \eqref{eq:BreezyGem} which only depend on a single integer prime $p$ are deduced in our proofs by a decomposition in terms of the infinitely many maximal ideals of $\bbZ[x]$ that contain $p$, and similar issues may be expected to arise.    
These challenges also remain open in the simpler setting of walk matrices of directed graphs; see \cite[Conjecture 1.4]{van2025cokernel}. 

Let us emphasize that the preceding discussion concerns universality with respect to the distribution of the entries. 
Structural changes to the random matrix, on the other hand, will often change the universality class. 
As was mentioned earlier, the prediction in \eqref{eq: DetWalkPsquare} explicitly differs from our prior results in \cite[Theorem 1.2]{van2025cokernel} for walk matrices in a setting without symmetry constraint. 
That Conjectures \ref{conj: Discriminant} and \ref{conj: WalkMatrixSquareFree} are stated for random graphs with self-loops is also no coincidence: the considered statistics are sensitive to the presence of a diagonal in the matrix. 
For instance, for graphs without self loops, it is known that the discriminant is always even \cite[Section 5]{wang2016square}.

In the context of \Cref{conj: WalkMatrixSquareFree}, one may also wonder what happens if the vector $\zeta$ is taken to be deterministic instead of random. 
The all-ones vector $\zeta=(1,\ldots,1)^{\T}$ is particularly relevant, as this special case of $\Phi$-cospectrality has special historical interest. 
This case dates back to the 1980 work of Johnson and Newman \cite{johnson1980note}, who viewed it as a natural generalization of adjacency cospectrality where the entries $0$ and $1$ in the adjacency matrix are replaced by formal symbols $x,y\in \bbR$.    
It is precisely this notion that was considered in the original conditions by Wang and Xu \cite{wang2006sufficient}.
Remarkably, numerical evidence presented in \Cref{sec: Conclusion} seems to suggest that the prediction in \eqref{eq: DetWalkPsquare} for the probability that $p^2$ divides the determinant is universal in terms of the vector $\zeta$ for every deterministic $\zeta$ whose reduction modulo $p$ is nonzero, \emph{except} for the all-ones vector when $p=2$.

While the aforementioned exceptional phenomena surrounding the prime $2$ go beyond the specific models studied here, they are not fundamentally beyond our two-step framework. 
We intend to pursue the relevant modifications of our framework in a future work.   

\subsection{Outline}\label{sec: IntroOutline}
\Cref{sec: Background} provides background, such as that surrounding profinite completions. 
\Cref{sec: Rephrasing} establishes equivalent phrasings of the conditions from \Cref{thm: WalkSufficientConditionWithLoops,thm: DiscriminantCondition} in terms of $\bbZ[x]$-module structure.
We subsequently study the relevant $\bbZ[x]$-modules in a profinite random matrix ensemble in \Cref{sec: ComputationAnalytical}.
Our main results are \Cref{thm: HaarMainWalk,thm: MainHaarDisc}.
We conclude in \Cref{sec: Conclusion}. 

The subsections concerning the walk matrix and discriminant in \Cref{sec: Rephrasing,sec: ComputationAnalytical} can be read separately. 
While these parts are conceptually related, their logical treatment is mostly self-contained.

\section{Background and notation}\label{sec: Background}
The most natural way to state our subsequent rigorous results relies on concepts from arithmetic statistics and commutative algebra that are not so common in the areas of graph theory where the spectral characterization topic comes from.
We here provide the required background.  

\subsection{Definition and interpretation of cokernels}\label{sec: PrelimCokernel}
In general, if $R$ is a commutative ring and we are given a matrix $\bX \in R^{n\times m}$ then the \emph{cokernel} is the quotient $R$-module that measures to what extent $\bX$ fails to be surjective as an $R$-module morphism onto $R^n$: 
\begin{align}
    \coker(\bX) \de R^n / \operatorname{Im}(\bX) \ \ \text{ where }\ \ \operatorname{Im}(\bX) \de \bigl\{ \bX v: v\in R^m  \bigr\}. \label{eq: Def_Coker} 
\end{align}
Given a matrix $\bZ \in R^{n\times k}$ we denote $\coker(\bX,\, \bZ)$ for the cokernel of the $n\times (m+k)$ rectangular matrix found by concatenating $\bX$ and $\bZ$.
If $z \in R^{n}$ is a vector then we similarly write $\coker(\bX,\, z)$. 

Now introduce a formal symbol $x$ and let $R = \bbZ[x]$.
Given an integer matrix $\bM\in \bbZ^{n\times n}$ it then holds that $\bM - x\bI \in \bbZ[x]^{n\times n}$ where $\bI\in \bbZ^{n\times n}$ is the identity matrix.  
In particular, we can consider the $\bbZ[x]$-module $\coker(\bM -x\bI)$.
To interpret this object, it is instructive to compare with another natural module associated to $\bM$. 
Consider the $\bbZ[x]$-module structure on $\bbZ^n$ induced by the action of the matrix. 
That is, let $Q(x) v \de Q(\bM) v$ for every $Q(x) \in \bbZ[x]$ and $v\in \bbZ^n$.
One then has a natural map $\bbZ^n \to \operatorname{coker}(\bM -x\bI)$
found by composing the embedding $\bbZ^n \to \bbZ[x]^n$ with the quotient map: 
\[\begin{tikzcd}
    {\bbZ^n} && {\bbZ[x]^n} && {\operatorname{coker}(\mathbf{M}-x\mathbf{I})} 
    \arrow[from=1-1, to=1-3]
    \arrow[bend right=20, from=1-1, to=1-5]
    \arrow[from=1-3, to=1-5]
\end{tikzcd}\]
One can check that this composition defines a bijection and is compatible with the $\bbZ[x]$-module structures on $\bbZ^n$ and $\operatorname{coker}(\bM-x\bI)$.  
In other words, the $\bbZ[x]$-modules are isomorphic.

\subsection{Profinite completion and its Haar distribution}\label{sec: ProfiniteCompletion}
Recall that we will take matrix entries from the uniform probability distribution on some larger space than $\{0,1\}$ to achieve analytical tractability. 
One natural idea would be to take $\bbZ$ or $\bbZ[x]$ as this space, but these are infinite countable sets and hence do not admit a uniform probability distribution.
Finite quotients of these rings, on the other hand, do admit a uniform probability distribution and the distributions on different quotients are compatible in the sense that, for example, the uniform law on $\bbZ/20\bbZ$ pushes forward to the uniform law on $\bbZ/10\bbZ$ under the quotient map.
This leads one to the \emph{profinite completion}, which is an object that gathers all finite quotients and comes with a natural probability distribution.

Consider a commutative ring $R$ and let $\cF \de \{I \trianglelefteq R: \# R/I < \infty\}$ be the ideals of finite index. 
For any $I,J \in \cF$ with $I \subseteq J$ let $\varphi_{I,J}: R/I \to R/J$ be the quotient morphism. 
The \emph{profinite completion of $R$} is the inverse limit of this system, which can be concretely realized defined as the subring of $\prod_{I \in \cF} R/I$ consisting of only those sequences that are compatible with the quotient morphisms $\varphi_{I,J}$: 
\begin{align}
    \widehat{R} \de \bigl\{(r_I)_{I\in \cF} \in \textstyle\prod_{I \in \cF} R/I: \varphi_{I,J}(r_I) = r_J ,\ \forall I \subseteq J    \bigr\}.\label{eq: Def_ProfiniteCompletion}
\end{align}
We equip $\widehat{R}$ with the subspace topology from $\prod_{I\in \cF}R/I$ where the latter carries the product topology of the discrete sets $R/I$.   
This makes $(\widehat{R},+)$ into a compact Hausdorff topological group. 
In particular, and this is the crucial fact for our purposes, there exists a unique Haar probability measure on $\widehat{R}$. 
The translation invariance of the latter implies that for every $I\in \cF$ and $S\subseteq  R/I$,
\begin{align}
        \operatorname{Haar}( \{(r_J)_{J\in \cF}: r_I \in S\}) = \frac{\#S}{\#R/I}.\label{eq: Def_HaarProperty} 
\end{align}
In other words, the Haar measure pushes forward to the uniform law on the finite quotients of $R/I$ under the projection. 
This recovers the aforementioned intuition that the profinite completion is the natural object that gathers all the finite quotients and their uniform laws.
We refer to \cite{ribes2017profinite,wilkes2024profinite} for additional background on profinite completions.

\begin{remark}
    Another way to study the properties of a random elements of $\bbZ$ is to consider the limiting behavior of the uniform probability measures $\mu_n$ on $[-n,n]\cap \bbZ$. 
    This perspectives is equivalent to considering the Haar measure on $\widehat{\bbZ}$ for properties that only depends on the residue of the random integer modulo some fixed $k$, due to \eqref{eq: Def_HaarProperty}.  
    What makes the profinite completion particularly convenient, however, is that its Haar distribution remains tractable for properties depending on infinitely many different residues. 
    This is more delicate for $\mu_n$ since the law is far from uniform on $\bbZ/k_n\bbZ$ for large $k_n\gg n$.  
    Analogous comments apply to $\bbZ[x]$ when one considers random polynomials with sufficiently uniform coefficients and diverging degree. 
\end{remark}

\subsubsection{On the special cases \texorpdfstring{$\bbZ$}{Z} and \texorpdfstring{$\bbZ[x]$}{Z[x]}}\label{sec: SpecialCaseProfiniteZ_Z[x]}

Our subsequent results will concern the profinite completions of specific rings, particularly $\bbZ[x]$, and more concrete descriptions can be given.
If one prefers, one could alternatively adopt the following equivalent formulations as the definitions and simply remember from the above why these are natural settings. 

For $R = \bbZ$, it follows from the Chinese remainder theorem that $\widehat{\bbZ} \cong \prod_{\textnormal{primes } p} \bbZ_p$ where $\bbZ_p$ is the ring of \emph{p-adic integers} associated to the prime $p$.
The $p$-adic integers can be identified with formal power series in the variable $p$ with coefficients in $\{0,\ldots,p-1 \}$: 
\begin{equation}
    \bbZ_p \de \bigl\{ \textstyle\sum_{i=0}^\infty c_i p^i : c_i\in\{0,1,\ldots,p-1 \} \bigr\}.\label{eq: Def_Zp} 
\end{equation}
Addition and multiplication is done with overflow on coefficients being carried to higher powers of $p$.
The $p$-adic integers admit a unique Haar probability measure that corresponds to sampling independent uniform random coefficients: it holds for $H_p \sim \operatorname{Haar}(\bbZ_p) $ that
\begin{align}
    H_p =\textstyle \sum_{i=0}^\infty C_i p^i \quad \text{where}\quad C_i \sim \operatorname{Unif}\{0,1,..,p-1\}.  
\end{align}
Under the isomorphism $\widehat{\bbZ} \cong \prod_p\bbZ_p$, a Haar random element of $\widehat{\bbZ}$ corresponds to sampling independent elements from the Haar measures on the different factors:
\begin{align}
    \operatorname{Haar}\bigl(\widehat{\bbZ} \bigr) = \otimes_{\textnormal{primes }p} \operatorname{Haar}(\bbZ_p).\label{eq: HaarProfiniteZ}   
\end{align}
The reader is referred to the books \cite{gouvea2020padicintro,katok2007p,robert2000course} for additional background.

In greater generality, the role of the $p$-adic integers may be replaced by the $\mathfrak{m}$-adic completions\footnote{
Here, the $\mathfrak{m}$-adic completion of a ring is defined similar to \eqref{eq: Def_ProfiniteCompletion} except that we now consider the inverse limit of the system $R/\mathfrak{m}^n$.
We refer to \cite[Chapter 10]{atiyah2018introduction} and \cite[Chapter 7]{eisenbud2013commutative} for additional background.
} at maximal ideals $\mathfrak{m}\subseteq R$. 
(See  \eg \cite[\S 2.1]{jaikin2023finite} for such a statement for arbitrary finitely generated rings.)      
The maximal ideals of $\bbZ[x]$ are $\mathfrak{m} = p\bbZ[x] + \beta(x)\bbZ[x]$ with $p\in \bbZ$ a prime and $\beta\in \bbZ[x]$ a monic polynomial of degree $\geq 1$ whose reduction modulo $p$ is irreducible in $\bbF_p[x]$; see \cite[p.22]{reid1995undergraduate}. 
Thus, one has the following explicit description in the case of $R = \bbZ[x]$:
\begin{align}
    \widehat{R} \cong \textstyle \prod_{p}\prod_{ \beta(x)} R_{p,\beta}\quad \textnormal{ with }\quad   R_{p,\beta} \de  \bigl\{ \sum_{i=0}^\infty c_i \beta^i: c_i\in \bbZ_p[x] \textnormal{ with } \deg(c_i)<\deg(\beta) \bigr\},\label{eq: ProfiniteZpbeta}  
\end{align}
where the products correspond to the distinct maximal ideals, meaning that the second product has as many factors as there are irreducible monic polynomial in $\bbF_p[x]$. 
Multiplication in $R_{p,\beta}$ is done with overflow on the degree of the polynomial coefficients being carried to higher powers of $\beta$. 
(In other words, $R_{p,\beta}$ is isomorphic as a ring to the quotient of $\bbZ_p[x][\![y]\!]$ by the ideal generated by $y-\beta(x)$.)   
A random element $H_{p,\beta} \sim\operatorname{Haar}(R_{p,\beta})$  can be generated using independent coefficients: 
\begin{align}
    H_{p,\beta} = \textstyle \sum_{i=0}^\infty  c_i  \beta^i \  \textnormal{ where }\  c_i = \sum_{j=0}^{\deg(\beta)-1} C_{i,j} x^j \ \textnormal{ with }\  C_{i,j} \sim \operatorname{Haar}(\bbZ_p).  \label{eq: HaarElementZpbeta}
\end{align}
A Haar random element under the isomorphism $\widehat{R} \cong \prod_{p, \beta(x)} R_{p,\beta}$ corresponds to sampling independent elements from the Haar measures on the different factors, similar to \eqref{eq: HaarProfiniteZ}.

\subsection{Pieces of modules and relevant properties}\label{sec: PrelimOnModules}
Given an Abelian group $G$ and a $\bbZ[x]$-module $\sM$ we introduce the following notation: 
\begin{align}
    G_p \de G\otimes_{\bbZ} \bbZ_p \quad \textnormal{ and }\quad \sM_{p,\beta} \de \sM \otimes_{\bbZ[x]} R_{p,\beta}.\label{eq: Def_NotationGpMpbeta}   
\end{align}
Here, $\oplus_{\bbZ}$ and $\otimes_{\bbZ[x]}$ refer to tensor products of Abelian groups or $\bbZ[x]$-modules, respectively. 
The subscript $\bbZ$ reflects that Abelian groups are exactly the same object as $\bbZ$-modules.  

Concrete calculations for the group operation can be done using that every finitely generated Abelian group is a direct sum of cyclic groups by the classification theorem together with the following standard fact:  
\begin{lemma}\label{lem: GroupGp}
     Suppose that $G \cong \oplus_{i=1}^m (\bbZ/k_i\bbZ)$ for $k_1,\ldots,k_m \geq 0$. 
     Then, $G_p \cong  \oplus_{i=1}^m  (\bbZ/p^{e_i}\bbZ)$ with $e_i \de \sup\{j \geq 0 : p^j\mid k_i  \}$. 
     Here, it is to be understood that $\bbZ/p^\infty\bbZ \de \bbZ_p$ when $k_i = 0$. 
\end{lemma}
\begin{proof}
    Note that  $(\bbZ/p^{e_i} \bbZ)\otimes_{\bbZ} \bbZ_p \cong \bbZ/p^{e_i}\bbZ$. 
    Further, any integer $k'\geq 1$ that is coprime with $p$ is invertible in $\bbZ_p$ \cite[\S 1.5]{robert2000course}, implying that $(\bbZ/k'\bbZ)\otimes_{\bbZ} \bbZ_p = 0$ if $k' \not\equiv 0\bmod p$.
    The claim now follows from the Chinese remainder theorem since tensor products distribute over direct sums.      
\end{proof}
In particular, \Cref{lem: GroupGp} implies that every finite Abelian group satisfies $G\cong \oplus_{\textnormal{primes }p} G_p$. 
A related decomposition result for modules that are finitely generated as groups (but not necessarily finite) is given by the following \Cref{lem: ModuleMpMpbeta}: 
\begin{lemma}\label{lem: ModuleMpMpbeta}
    Consider a $\bbZ[x]$-module $\sM$ which is finitely generated as an Abelian group. 
    Then, it holds for every prime $p$ that $\sM_p \cong \oplus_{\beta} \sM_{p,\beta}$ as $\bbZ[x]$-modules, where the direct sum runs over representatives in $\bbZ[x]$ of the irreducible monic polynomials in $\bbF_p[x]$.   
\end{lemma}

A proof for \Cref{lem: ModuleMpMpbeta} can be found in \Cref{apx: proofModuleMpMpbeta}. 
\section{Rephrasing conditions using the \texorpdfstring{$\bbZ[x]$}{Z[x]}-module structure of \texorpdfstring{$\coker(\bM-x\bI)$}{coker(M-xI)}}\label{sec: Rephrasing}
We now consider the first step of the approach outlined in \Cref{sec: IntroTwoStep} and rephrase the sufficient conditions in terms of module-theoretic data. 
The condition from \Cref{thm: WalkSufficientConditionWithLoops} is rephrased in \Cref{prop: ModuleRephrasingWalkDet} and the condition from \Cref{thm: DiscriminantCondition} is considered in \Cref{prop: ModuleRephrasingDiscriminantCondition,prop: ModuleRephrasingDiscriminantConditionTwo}.

Throughout this section we let $\zeta\in \bbZ^n$ be an integer vector and we consider an integer matrix $\bM\in \bbZ^{n\times n}$.
The results of the current section do not require $\bM$ to be symmetric.
The symbol $p$ will always refer to a prime number. 
\subsection{Regarding the walk matrix}  
Note that the walk matrix $\bW = [\zeta, \bM\zeta, \ldots, \bM^{n-1}\zeta]$ has integer entries.
In particular, since $\bbZ$-modules are just Abelian groups, we may consider $\coker(\bW)$ as a group. 
A group-theoretic rephrasing of the sufficient conditions similar to the following one was also considered in \cite{van2025cokernel}. 
Readers consulting that previous work are warned, however, that our notation $G_p$ from \eqref{eq: Def_NotationGpMpbeta} differs from the notation used in \cite{van2025cokernel}.  

\begin{lemma}\label{lem: psquared_to_cokergroup}
    It holds that $p^{2} \nmid \det(\bW)$ if and only if $\coker(\bW)_p$ is the trivial group or $\bbZ/p\bbZ$.      
\end{lemma}
\begin{proof}
    Let $\bW = \bU \bD \bV$ be the \emph{Smith normal form}.
    This means that $\bU, \bV\in \bbZ^{n\times n}$ are integer matrices with $\det(\bU),\det(\bV) \in \{-1,+1 \}$ and $\bD = \operatorname{diag}(d_1,\ldots,d_n)$ is a diagonal matrix with integer entries $d_i\in \bbZ$ satisfying $d_{i}\mid d_{i+1}$ for every $i\geq 1$.
    Note that $\det(\bW) = \pm \prod_{i=1}^n d_i$. 
    Consequently, it holds that $p^2 \nmid \det(\bW)$ if and only if $p\nmid d_{n-1}$ and $p^2 \nmid d_n$. 
    
    The constraint on the determinant of $\bU$ and $\bV$ means that these matrices are invertible over the integers.  
    In particular, $\operatorname{Im}(\bW) = \operatorname{Im}(\bU \bD)$. 
    Now, recalling the definition of the cokernel from \eqref{eq: Def_Coker},
    \begin{align}
        \textstyle\coker(\bW) = \bbZ^n / \operatorname{Im}(\bU \bD) \cong   \coker(\bD) \cong \oplus_{i=1}^n \bbZ/d_i \bbZ,  \label{eq:ShyQuote}
    \end{align}
    where the isomorphism $\bbZ^n / \operatorname{Im}(\bU \bD) \cong   \coker(\bD)$ uses that the map $z\mapsto \bU^{-1}z$ on $\bbZ^n$ sends $\operatorname{Im}(\bU \bD)$ to $\operatorname{Im}(\bD)$.
    Use \Cref{lem: GroupGp} to conclude that $p^{2}\nmid \det(\bW)$ if and only if $\coker(\bW)_p \in \{0,\bbZ/p\bbZ \}$. 
\end{proof}

Note that if $\bbZ^n$ is equipped with the $\bbZ[x]$-module structure induced by the action of $\bM$ as in \Cref{sec: PrelimCokernel} then the image of $\bW$ is exactly the $\bbZ[x]$-submodule generated by the vector $\zeta$: 
\begin{align}
    \textstyle \operatorname{Im}\bigl(\bW\bigr)  = \bigl\{\bW v: v\in \bbZ^n \bigr\} = \bigl\{ \sum_{i=0}^{n-1} v_i\bM^i  \zeta: v_{i}\in \bbZ\bigr\} = \bigl\{ Q(\bM)\zeta : Q\in \bbZ[x]   \bigr\}, 
\end{align}
where the final equality uses that $\bM^n = \sum_{i=0}^{n-1}c_i \bM^i$ for certain $c_i\in \bbZ$ due to the Cayley--Hamilton theorem. 
It follows that the quotient $\coker(\bW) = \bbZ^n / \operatorname{Im}(\bW)$ is not only an Abelian group but also canonically equipped with the structure of a $\bbZ[x]$-module. 
Moreover, as $\bbZ[x]$-modules,  
\begin{align}
    \coker\bigl(\bW\bigr) = \frac{\bbZ^n}{\operatorname{Im}(\bW)} \cong   \frac{\bbZ[x]^n }{\operatorname{Im}(\bM-x\bI) + \bbZ[x] \zeta } = \coker\bigl(\bM - x\bI,\, \zeta\bigr),\label{eq:DampQuote}   
\end{align}
where we consider $[\bM-x\bI, \zeta ]$ as a rectangular matrix over $\bbZ[x]$ in the final equality.

To understand why \eqref{eq:DampQuote} is useful, note that the entries of $\bW$ are nontrivial algebraic combinations of those of $\bM$ and $\zeta$ which would complicate direct study of $\coker(\bW)$.  
On the other hand, $\coker\bigl(\bM - x\bI,\, \zeta\bigr)$ refers directly to $\bM$ and $\zeta$ in a linear way.

We can refine the group-theoretic formulation from \Cref{lem: psquared_to_cokergroup} by considering the admissible $\bbZ[x]$-module structures: 
\pagebreak[3]
\begin{proposition}\label{prop: ModuleRephrasingWalkDet}
    It holds that $p^2 \nmid \det(\bW)$ if and only if one of the following two events occurs: 
    \begin{enumerate}[leftmargin = 1.6em]
        \item It holds that $\coker\bigl(\bM - x\bI,\, \zeta\bigr)_{p,\beta} \cong 0$ for all monic $\beta(x) \in \bbZ[x]$ with irreducible reduction in $\bbF_p[x]$.
        \item Or, there exists $a\in \bbZ$ such that one has an isomorphism of $\bbZ[x]$-modules 
        \begin{align}
            \coker\bigl(\bM - x\bI,\, \zeta\bigr)_{p, x-a} \cong \bbF_{p}[x]/(x-a)\bbF_{p}[x],
        \end{align}
        and it holds that $\coker\bigl(\bM - x\bI,\, \zeta\bigr)_{p,\beta} \cong 0$ for every $\beta \not\equiv x-a \bmod p$.     
    \end{enumerate}
\end{proposition}
\begin{proof}
    Note that a $\bbZ[x]$-module $\sM$ satisfies $\sM_p \cong \bbZ/p\bbZ$ as a group (resp. $\sM_p \cong 0$) if and only if $\sM_p \cong \bbF_p[x]/(x-a)\bbF_p[x]$ as a module for some $a\in \bbZ$ (resp. if and only if it is the zero module). 
    The result is hence immediate from \eqref{eq:DampQuote} and \Cref{lem: ModuleMpMpbeta,lem: psquared_to_cokergroup}.     
\end{proof} 
To understand how often the determinant of the walk matrix is square-free, it now suffices to study the $\bbZ[x]$-module $\coker(\bM-x\bI,\zeta)$.
This is done in \Cref{sec: ConditionWalk} for a model where $\bM$ and $\zeta$ are both taken to have Haar random entries on the profinite completion of $\bbZ[x]$.

\subsection{Regarding the discriminant}\label{sec: Discriminant}
We next rephrase the condition from \Cref{thm: DiscriminantCondition} based on the discriminant in terms of $\coker(\bM-x\bI)$. 
We start in \Cref{sec: PreliminariesDiscriminant} by rephrasing the event where $p^2$ divides the discriminant in terms of the factorization of the polynomial. 
A rephrasing in terms of the cokernel is next given in \Cref{sec: CokerDisc}.

\subsubsection{Preliminaries}\label{sec: PreliminariesDiscriminant}
Consider the $\bbZ$-module of polynomials with degree strictly smaller than $i\geq 1$: 
\begin{align}
    \sP_i \de \bigl\{ f(x)\in \bbZ[x] : \operatorname{deg}(f)  < i \bigr\}. 
\end{align}
Given $\phi,\psi \in \bbZ[x]$, we define the \emph{Sylvester map} to be the $\bbZ$-module morphism specified by  
\begin{align}
    \cS_{\phi,\psi} :\sP_{\operatorname{deg}(\psi)} \oplus  \sP_{\operatorname{deg}(\phi)} \to \sP_{\operatorname{\operatorname{deg}(\psi) + \operatorname{deg}(\phi) }}:  (f, g) \mapsto  f\phi + g\psi.   \label{eq: Def_Sylvester}
\end{align}
If one uses the monomials $1,x,x^2,\ldots$ as a basis for $\sP_{i}$ then one can represent $\cS_{\phi,\psi}$ as a square matrix whose entries are the coefficients of $\phi$ and $\psi$ with shifts from one row to the next; see \eg \cite[\S 2.3]{wang2016square}.  
The determinant in this basis is usually called the \emph{resultant of $\phi$ and $\psi$}: 
\begin{align}
    \operatorname{Res} \bigl(\phi,\psi\bigr) \de \det\bigl(\cS_{\phi,\psi} \bigr).\label{eq:MoistCity}
\end{align}
A different choice of $\bbZ$-module basis for the domain and codomain can change the determinant, but only by by a factor $\pm 1$. 
Ambiguity up to a unit will not matter for us, so we may simply refer to the determinant of the Sylvester map for a canonical perspective.    
The relevance for our purposes occurs when the second polynomial is the derivative of the first: 
\begin{definition}\label{def: Discriminant}
    Consider a monic polynomial $\phi\in \bbZ[x]$ with degree no less than one. 
    Then, the resultant of $\phi$ and its formal derivative $\phi'$ is called its \emph{discriminant} and is denoted $
        \Delta_{\phi} \de \operatorname{Res}(\phi, \phi')$. 
\end{definition}
It may be shown that \Cref{def: Discriminant} is equivalent to the formula in terms of roots given in \eqref{eq: DiscriminantRootFormula}. 
We will not need that formula but refer the interested reader to \cite[Chapter 12]{gelfand1994discriminantresultant}.

\Cref{def: Discriminant} implies that the discriminant of an integer polynomial is divisible by $p$ if and only if the Sylvester map of $\phi$ and $\phi'$ is singular over $\bbF_p$.
In this context, note that if  $\bbK$ is a field,  then B\'ezout's identity in principle ideal domain $\bbK[x]$ states that for every $\phi,\psi\in \bbK[x]$, 
\begin{align}
    \bigl\{ f(x)\phi(x) + g(x)\psi : f(x), g(x)\in \bbK[x]  \bigr\} = \bigl\{ h(x)\operatorname{gcd}(\phi,\psi): h(x)\in \bbK[x] \bigr\}. 
\end{align}
In particular, the Sylvester map is nonsingular modulo $p$  if and only if the polynomials $\phi$ and $\psi$ are coprime over $\bbF_p$. 
(We here use that a linear transformation is nonsingular if and only if it is surjective.)   
This yields the following classical \Cref{prop: DiscriminantDivisibleP}. 
A polynomial $\phi\in \bbF_p[x]$ is said to be \emph{square-free} if there does not exist any irreducible $\beta\in \bbF_p[x]$ with $\operatorname{deg}(\beta)\geq 1$ and $\beta^2 \mid \phi$.  
\begin{proposition}\label{prop: DiscriminantDivisibleP}
    Consider some monic $\phi \in \bbZ[x]$ with $\deg(\phi)\geq 1$. 
    Then, $p\nmid \Delta_{\phi}$ if and only if $\phi\bmod p$ is square-free over $\bbF_p$.    
\end{proposition}
\begin{proof}
    Let $\phi \equiv \prod_{i}\beta_i(x)^{n_i} \bmod p$ be the factorization of $\phi$ into powers of distinct irreducible monic polynomials $\beta_i\in  \bbF_p[x]$. 
    Then, $\phi' = \sum_{i}n_i \beta_i(x)^{n_i - 1}\beta_i'(x) \prod_{j\neq i} \beta_j(x)^{n_j}$ and we observe that $\beta_i \mid \phi'$ if and only if $n_i \geq 2$. 
    Thus, $\gcd(\phi,\phi')=1$ if and only if $\phi$ is not divisible by $\beta^2$ for all irreducible $\beta$. 
    In other words, the Sylvester map is nonsingular modulo $p$ if and only if $\phi$ is square-free. 
\end{proof}

An integer $d\in \bbZ$ is said to be \emph{exactly divisible by $p$}, denoted $p\parallel d$, if $p\mid d$ but $p^2\nmid d$. 
A proof of the following characterization of the event where the discriminant is exactly divisible by an odd prime can be found in work of Ash, Brackenhoff, and Zarrabi \cite[Proposition 6.7]{ash2007equality}.
One of the implications also appears in the proof of the sufficient condition \Cref{thm: DiscriminantCondition} by Wang and Yu \cite[Section 4]{wang2016square} with a different argument that relies on a direct analysis of the Sylvester map.
\pagebreak[3]
\begin{proposition}[{\cite{ash2007equality,wang2016square}}]\label{prop: OddSquaredDividesDiscriminant}
    Consider an odd prime $p$ and a monic polynomial $\phi \in \bbZ[x]$ with $\deg(\phi)\geq 1$. 
    Then, it holds that $p \parallel  \Delta_{\phi}$ if and only if there exists some $a \in \bbZ$ and a square-free polynomial $\xi \in \bbF_p[x]$ with $\xi(a) \not\equiv 0 \bmod p$ such that $\phi(x) \equiv (x-a)^2 \xi(x) \bmod p$ and $p^2\nmid \phi(a)$.  
\end{proposition}
\begin{remark}
    A priori, the event in \Cref{prop: OddSquaredDividesDiscriminant} depends on $a \bmod p^2$ but \cite[Proposition 6.7]{ash2007equality} shows that it actually only depends on the reduction modulo $p$. 
    That is, if $\phi(x) \equiv (x-a)^2 \xi(x) \bmod p$ for square-free $\xi$ with $\xi(a)\not\equiv0\bmod p$ and $p^2\nmid \phi(a)$, then $p^2 \nmid \phi(b)$ for all $b \equiv a \bmod p$.   
\end{remark}

\subsubsection{Connection to the cokernel}\label{sec: CokerDisc}
We next rephrase \Cref{prop: DiscriminantDivisibleP,prop: OddSquaredDividesDiscriminant} in terms of $\coker(\bM - x\bI)$ in \Cref{prop: ModuleRephrasingDiscriminantCondition,prop: ModuleRephrasingDiscriminantConditionTwo}, respectively. 
The proofs rely on the following \Cref{lem: DivLength}.

The \emph{length} of a module $\sN$ over a ring $R$ is the maximal length of a chain of proper submodules: 
\begin{align}
    \operatorname{Length}_R(\sN) \de \sup\{L\geq 1: \exists \sN_0 \subsetneq \sN_1 \subsetneq \cdots \subsetneq \sN_L \textnormal{ with }\sN_0 = 0 \textnormal{ and }\sN_L = \sN \}.\label{eq:YoungInk}
\end{align}
\begin{lemma}\label{lem: DivLength}
    Fix a monic polynomial $\beta \in \bbZ[x]$ with $\operatorname{deg}(\beta)\geq 1$ and irreducible reduction in $\bbF_p$.
    Then, 
    \begin{equation}
    \max\bigl\{k\geq 0: \beta(x)^k \mid \varphi_{\bM}(x) \bmod p \bigr\} =
    \operatorname{Length}_{R_{p,\beta}/pR_{p,\beta}}\biggl(
    \coker\Bigl(\frac{(\bM - x\bI)_{p,\beta}}{p 
            \coker(\bM - x\bI)_{p,\beta}}
            \Bigr) \biggr). \label{eq:UsefulRag}
    \end{equation}
\end{lemma}
\begin{proof} 
    Consider the factorization $\varphi_{\bM}(x) \equiv \prod_{i} \beta_i^{k_i} \bmod p$ of the characteristic polynomial into powers of distinct irreducible monic polynomials $\beta_i \in \bbF_p[x]$. 
    Then, the rational canonical form \cite[\S12.2]{dummit_foote_2004} of $\widebar{\bM} \de \bM \bmod p$ over $\bbF_p$ yields nonnegative integers  $\lambda_{i,j} \geq 0$  with $\sum_{j} \lambda_{i,j} = k_i$ and 
    \begin{equation}
        \frac{\coker\bigl(\bM - x\bI\bigr)}{p\coker\bigl(\bM - x\bI\bigr)}
        \cong 
        \coker\bigl(\widebar{\bM}-x\bI\bigr) 
        \cong 
        \prod_{i}\prod_{j} \frac{\bbF_p[x]}{\beta_i^{\lambda_{i,j}}\bbF_p[x]}.\label{eq:GiddyLobster}
    \end{equation}
    It follows directly from the definition of $R_{p,\beta}$ that $(\bbF_p[x]/\beta^{\lambda} \bbF_p[x]) \otimes_{\bbZ[x]} R_{p,\beta} = \bbF_p[x]/\beta^{\lambda} \bbF_p[x]$.
    Further, $(\bbF_p[x]/\gamma^{\lambda} \bbF_p[x])\otimes_{\bbZ[x]} R_{p,\beta}  = 0$ for $\gamma$ with coprime reduction to $\beta$.
    (The latter can be shown using that $\beta^\lambda$ is then invertible in $R_{p,\gamma}$, or as a consequence of the previous sentence with $\gamma = \beta$ and \Cref{lem: ModuleMpMpbeta}.) 
    Hence, taking a tensor product with $R_{p,\beta_i}$ in \eqref{eq:GiddyLobster}, 
    \begin{equation}
        \frac{\coker\bigl(\bM - x\bI\bigr)_{p,\beta_i}}{ p
            \coker\bigl(\bM - x\bI\bigr)_{p,\beta_i}} 
            \cong
        \prod_{j} \frac{\bbF_p[x]}{\beta_i^{\lambda_{i,j}}\bbF_p[x]}. \label{eq:UnripeUser}
\end{equation}
Note that $\operatorname{Length}_{R}(\prod_{j} \sN_j) = \sum_j \operatorname{Length}_{R}( \sN_j)$ for any finite product of $R$-modules. 
Further, the submodules of $\sN \de \bbF_p[x]/\beta^{\lambda} \bbF_p$ are $\beta^{j} \sN$ and consequently $\operatorname{Length}_{R_{p,\beta}/pR_{p,\beta}}(\sN) = \lambda$.  
Hence, using that $\sum_{j}\lambda_{i,j}=k_i$ with $k_i$ the power of $\beta_i$ in the factorization of $\varphi_{\bM}$ yields \eqref{eq:UsefulRag}.  
\end{proof}

\begin{proposition}\label{prop: ModuleRephrasingDiscriminantCondition}
    It holds that $p \nmid \Delta_{\bM}$ if and only if it holds for all monic $\beta(x) \in \bbZ[x]$ with irreducible reduction in $\bbF_p[x]$ that, as $\bbZ[x]$-modules,  
        \begin{equation}
            \frac{\coker\bigl(\bM - x\bI)_{p,\beta}}{p \coker\bigl(\bM - x\bI)_{p,\beta}} \cong 0 
            \qquad \textnormal{ or }\qquad 
            \frac{\coker\bigl(\bM - x\bI)_{p,\beta}}{p \coker\bigl(\bM - x\bI)_{p,\beta}}  \cong \frac{\bbF_p[x]}{\beta(x)\bbF_p[x] }. \label{eq:RipeRock}
        \end{equation}
\end{proposition}
\begin{proof}
    Recall from \Cref{prop: DiscriminantDivisibleP} that $p\nmid \Delta_{\bM}$ if and only if $\beta^2 \nmid \varphi_{\bM} \bmod p$ for all irreducible $\beta \in \bbF_p[x]$.  
    The result is hence immediate from \Cref{lem: DivLength}. 
    Indeed, recall that $\coker(\bM-x\bI)_{p,\beta}/p\coker(\bM - x\bI)_{p,\beta}$ has the form \eqref{eq:UnripeUser} and note that the modules of length $<2$ are those in \eqref{eq:RipeRock}.
\end{proof}

\begin{lemma}
\label{lem: ExactPDivisibility}
 It holds that $p\parallel \det(\bM-a\bI)$ for $a\in \bbZ$ if and only if as $\bbZ[x]$-modules, 
    \begin{equation}
            \frac{\coker\bigl(\bM - x\bI)_{p,x-a}
            }{
            (x-a)
            \coker\bigl(\bM - x\bI)_{p,x-a}} \cong \frac{\bbF_{p}[x]}{(x-a)\bbF_{p}[x]}\label{eq:VagueInk}
    \end{equation}
\end{lemma}
\begin{proof}
    It holds for any integer matrix $\bN \in \bbZ^{n\times n}$ that $\lvert \det(\bN)  \rvert=  \# \coker(\bN)$. 
    Indeed, this may be verified by considering the Smith normal form, similar to the proof of \Cref{lem: psquared_to_cokergroup}. 
    Hence, it follows from \Cref{lem: GroupGp} that $p\parallel \det(\bN)$ if and only if $\coker(\bN)_p \cong \bbF_p$ as a group.  
    It remains to show that $\coker(\bM-a\bI)_p \cong \bbF_p$ as groups if and only if \eqref{eq:VagueInk} holds. 

    Note that we have an isomorphism of groups $
        \coker(\bM - a\bI)
        \cong
            \coker(\bM - x\bI )
           /
            (x-a)
            \coker(\bM - x\bI )
            $.
        Hence, using the associativity of tensor products,
    \begin{equation}
       \coker\bigl(\bM - a\bI\bigr)\otimes_{\bbZ} \bbZ_p 
        \cong\coker\bigl(\bM - x\bI\bigr) \otimes_{\bbZ[x]} \frac{\bbZ_p[x]}{(x-a)\bbZ_p[x]}. \label{eq:EasyRock}
    \end{equation}
    It follows directly from the definition \eqref{eq: ProfiniteZpbeta} that $\bbZ_p[x]/(x-a)\bbZ_p[x]$ is a quotient of $R_{p,x-a}$. 
    In particular, $\bbZ_p[x]/(x-a)\bbZ_p[x] \cong R_{p,x-a}\otimes_{R_{p,x-a}}(\bbZ_p[x]/(x-a)\bbZ_p[x])$.  
    Hence, again using the associativity of tensor products and recalling the definition \eqref{eq: Def_NotationGpMpbeta}, 
    \begin{align}
        \coker(\bM - x\bI) \otimes_{\bbZ[x]} \frac{\bbZ_p[x]}{(x-a)\bbZ_p[x]} \cong \coker(\bM-x\bI)_{p,x-a} \otimes_{R_{p,x-a}}  \frac{\bbZ_p[x]}{(x-a)\bbZ_p[x]}.\label{eq:ValidSet} 
    \end{align}
    Combining \eqref{eq:EasyRock} and \eqref{eq:ValidSet} with the definition \eqref{eq: Def_NotationGpMpbeta}, we have the following isomorphism of groups: 
    \begin{align}
         \coker\bigl(\bM - a\bI)_p \cong   \frac{\coker(\bM-x\bI)_{p,x-a} }{(x-a)\coker(\bM-x\bI)_{p,x-a} }.\label{eq:LoudCamel} 
    \end{align}
    Note that the unique $\bbZ[x]$-module structure on the group $\bbF_p$ such that $x$ acts as multiplication by $a$ is given by $\bbF_p[x]/(x-a)\bbF_p[x]$. 
    It hence follows from \eqref{eq:LoudCamel} that $\coker(\bM-a\bI)_p \cong \bbF_p$ as a group if and only if \eqref{eq:VagueInk} occurs, concluding the proof. 
\end{proof}

    \begin{proposition}\label{prop: ModuleRephrasingDiscriminantConditionTwo}
     Consider an odd prime $p$. 
     Then, it holds that $p  \parallel \Delta_{\bM}$ if and only if there exists $a\in \bbZ$ such that \eqref{eq:RipeRock} is satisfied for every $\beta \not\equiv x - a \bmod p$ and it holds that, as $\bbZ[x]$-modules, 
        \begin{equation}
            \frac{\coker\bigl(\bM - x\bI)_{p,x-a}
            }{
            p
            \coker\bigl(\bM - x\bI)_{p,x-a}} \cong \frac{\bbF_{p}[x]}{(x-a)^2\bbF_{p}[x]}  \ \ \textnormal{ and }\ \ \frac{\coker\bigl(\bM - x\bI)_{p,x-a}
            }{
            (x-a)
            \coker\bigl(\bM - x\bI)_{p,x-a}} \cong \frac{\bbF_{p}[x]}{(x-a)\bbF_{p}[x]}.
            \label{eq:RipeRockTwo}
        \end{equation}
     \end{proposition}
\begin{proof}  
    Recall that \Cref{prop: OddSquaredDividesDiscriminant} states that $p\parallel \Delta_{\bM}$ if and only if there exists $a\in \bbZ$ with $p\parallel \varphi_{\bM}(a)$ such that $\varphi_{\bM}(x) \equiv (x-a)^2 \xi(x) \bmod p$ for square-free $\xi \in \bbF_p[x]$ with $\xi(a)\not\equiv 0 \bmod p$. 
    Here, as in the proof of \Cref{prop: ModuleRephrasingDiscriminantCondition} we have that \eqref{eq:RipeRock} holds if and only if $\beta^2 \nmid \varphi_{\bM} \bmod p$. 
    It hence remains to argue that \eqref{eq:RipeRockTwo} is equivalent to having $p\parallel \varphi_{\bM}(a)$ and $\max\{k\geq 0: (x-a)^k \mid \varphi_{\bM}\bmod p \}= 2$.

    \Cref{lem: DivLength} yields that $\max\{k\geq 0: (x-a)^k \mid \varphi_{\bM}\bmod p \}= 2$ if and only if 
    \begin{equation}
    \label{eqn: ModulesLength2}
            \frac{\coker\bigl(\bM - x\bI)_{p,x-a}}{p \coker\bigl(\bM - x\bI)_{p,x-a}} \cong \left(\frac{\bbF_p[x]}{(x-a)\bbF_p[x] }
            \right)^2
            \quad \textnormal{ or }\quad 
            \frac{\coker\bigl(\bM - x\bI)_{p,x-a}}{p \coker\bigl(\bM - x\bI)_{p,x-a}}  \cong \frac{\bbF_p[x]}{(x-a)^2\bbF_p[x] }. 
        \end{equation}
        Moreover, recall from \Cref{lem: ExactPDivisibility} that $p\parallel \varphi_{\bM}(a)$ if and only if 
        \begin{equation}
            \frac{\coker\bigl(\bM - x\bI)_{p,x-a}
            }{
            (x-a)
            \coker\bigl(\bM - x\bI)_{p,x-a}} \cong \frac{\bbF_{p}[x]}{(x-a)\bbF_{p}[x]}.\label{eq:WittyKnot}
        \end{equation} 
        It hence remains to show that a $R_{p,x-a}$-module $N$ cannot simultaneously satisfy $N/pN \cong (\bbF[x]/(x-a)\bbF_p[x])^2$ and $N/(x-a)N \cong \bbF_p[x]/(x-a)\bbF_p[x]$. 
        Indeed, if such a module were to exist, then taking additional quotients would yield that $N/(pN + (x-a)N) \cong (\bbF[x]/(x-a)\bbF_p[x])^2$ and $N/(x-a)N$ $N/(pN + (x-a)N) \cong \bbF_p[x]/(x-a)\bbF_p[x]$, a contradiction. 
        This concludes the proof. 
\end{proof}

\begin{remark}\label{rem: CokernelDisc}
    There are also different ways to relate the discriminant to cokernels: 
    \begin{enumerate}
        \item It follows from \eqref{eq:MoistCity} and \Cref{def: Discriminant} that $\lvert \Delta_{\phi} \rvert =  \#\coker(\cS_{\phi, \phi'})$. 
        \item It follows from \eqref{eq: DiscriminantRootFormula} that $\Delta_{\bM} = \det(\varphi_{\bM}'(\bM))$ with $\varphi_{\bM}'\in \bbZ[x]$ the derivative of the characteristic polynomial. 
        This implies that $\lvert \Delta_{\bM} \rvert =  \# \coker(\varphi_{\bM}'(\bM))$. 
    \end{enumerate}
    It unclear whether these alternative perspectives are helpful in probabilistic study.
    For instance, while the law of $\coker(P(\bM))$ with $P\in \bbZ[x]$ deterministic has been studied in the previous works (recall \Cref{sec: IntroUniversality}), the polynomial $\varphi_{\bM}'$ has nontrivial dependence on $\bM$.
    \Cref{prop: ModuleRephrasingDiscriminantCondition,prop: ModuleRephrasingDiscriminantConditionTwo} have the advantage that they depend on $\bM$ in a linear way.
\end{remark}

\section{Module statistics for profinite random matrix ensembles}\label{sec: ComputationAnalytical}

We found in \Cref{sec: Rephrasing} that a unifying feature of the sufficient conditions in Theorems \ref{thm: WalkSufficientConditionWithLoops} and \ref{thm: DiscriminantCondition} is that they can both be rephrased in the $\bbZ[x]$-module structure of $\coker(\bM-x\bI)$. 
To understand the satisfaction frequency of the conditions, it hence remains to understand the distribution of the module when $\bM$ is a random matrix. 
We here do this for a matrix ensemble with entries Haar distributed on the profinite completion of $\bbZ[x]$. 
Our main results are \Cref{thm: HaarMainWalk,thm: MainHaarDisc}.

\subsection{Model definition}\label{sec: ModelDefinition}
Throughout this section, we denote $R \de \bbZ[x]$. 
Recall the definition of the profinite completion $\widehat{R}$ and its Haar distribution from \Cref{sec: Background}.

Denote $\operatorname{SymHaar}(\widehat{R}^{n\times n})$ for the probability distributions on symmetric random matrices $\bM$ with values in $\widehat{R}^{n\times n}$ whose entries are Haar distributed and independent up to the symmetry constraint. 
That is, $\bM = \bM^{\T}$ and the upper-triangle satisfies
\begin{align}
    \bbP\Bigl(\bM_{i,j} \in \cE_{i,j}: \forall i \leq j \Bigr) = \prod_{i\leq j} \bbP\Bigl( \operatorname{Haar}\bigl(\widehat{R}\bigr) \in \cE_{i,j}\Bigr)\label{eq: DefSymHaar}
\end{align}
for all measurable subsets $\cE_{i,j} \subseteq \widehat{R}$.
We shall study the $\widehat{R}$-module $\coker(\bM - x\bI)$ as a model for the modules appearing in \Cref{prop: ModuleRephrasingWalkDet,prop: ModuleRephrasingDiscriminantCondition,prop: ModuleRephrasingDiscriminantConditionTwo}.
Thus, instead of a matrix with entries from $\bbZ$ or $\{0,1 \}$ we now consider entries from profinite completion of $\bbZ[x]$.    
Most crucially, the model $\bM$ retains the symmetry constraint that is necessary for the sufficient conditions for spectral characterization of integer matrices in \Cref{thm: WalkSufficientConditionWithLoops,thm: DiscriminantCondition}.

Given that the Haar measure is preserved by additive shifts, the matrix $\bM - x \bI$ has exactly the same distribution as the original matrix $\bM$. 
In particular, the cokernels have the same distributions:
\begin{equation}
    \bbP\bigl(\coker(\bM -x\bI  ) \in \cdot \bigr)  = \bbP\bigl(\coker(\bM   ) \in \cdot \bigr)\ \textnormal{ if }\ \bM \sim \operatorname{SymHaar}(\widehat{R}).\label{eq:WeepyJar}
\end{equation}
Recall from \Cref{sec: SpecialCaseProfiniteZ_Z[x]} that $\widehat{R} \cong \prod_{p}\prod_{\beta(x)}R_{p,\beta}$.
Let $\bM_{p,\beta}\in R_{p,\beta}^{n\times n}$ be the matrices that are found by applying this isomorphism entry-wise to $\bM$.
Then, since the cokernel of a matrix over a product ring always decomposes as a product corresponding to the factors,  
\begin{equation}
    \textstyle\coker(\bM) \cong \prod_{p}\prod_{\beta(x)}  \coker\bigl(\bM_{p,\beta} \bigr)\ \textnormal{ where }\ \bM_{p,\beta} \sim \operatorname{SymHaar}(R_{p,\beta}^{n\times n}), \label{eq:BoldAnt}   
\end{equation}  
where the isomorphism is as $\widehat{R}$-modules. 
It here holds that $\coker\bigl(\bM_{p,\beta}\bigr) \cong \coker(\bM )\otimes_{\widehat{R}} R_{p,\beta}$ so that the pieces in the decomposition \eqref{eq:BoldAnt} are analogous to the pieces of $\bbZ[x]$-modules defined in \eqref{eq: Def_NotationGpMpbeta}, except that the tensor product is now over the profinite completion instead of over $\bbZ[x]$. 
Our goal is to determine the probability that the constraints in \Cref{prop: ModuleRephrasingWalkDet,prop: ModuleRephrasingDiscriminantCondition,prop: ModuleRephrasingDiscriminantConditionTwo} are satisfied when the modules $\coker(\cdot)_{p,\beta}$ are formally replaced by the corresponding pieces in \eqref{eq:BoldAnt}. 

\begin{remark}\label{rem: AlternativeConjectureWithoutSymmetry}
    Removing the symmetry constraint from the profinite Haar model can also be used to give an alternative justification for a conjecture in \cite{van2025cokernel} concerning walk matrix statistics without symmetry constraint.
    That argument does not directly contribute to our main goal, as the symmetry constraint is essential for \Cref{thm: HaarMainWalk,thm: MainHaarDisc} to apply, except that the fact that one can recover previous conjectures may give some additional confidence in the modeling approach of the present paper. 
    We give details on the alternative argument for \cite[Conjecture 1.4]{van2025cokernel} in \Cref{sec: ConjWithoutSymmetry}. 
\end{remark}

\begin{remark}
\label{rem: Ix clarifying remark}
   While the translational invariance of the Haar distribution in \eqref{eq:WeepyJar} also applies for any other symmetric shift, we expect that the specific ocurrence of $x\bI$ in $\coker(\bM - x\bI)$ will be important in universality principles that allow different laws for the entries of $\bM$. 
   For example, Kahn and Koml\'os \cite[Theorem 1.3]{kahn2001singularity} show that the rank of a random matrix over a finite field is only universal if the entries are not in a proper affine subfield. 
   The latter implies that $\coker(\bM -x\bJ) \otimes_{\widehat{R}} \bbF_p[x]/\beta\bbF_p[x]$ with $\bJ$ the all-ones matrix is \emph{not} universal when $\bM$ has integer entries without symmetry constraint since $\bM - x\bJ \bmod p, \beta$ then has all entries in the same affine subfield $\bbF_p + x$ of $\bbF_p[x]/\beta \bbF_p[x]$, which is proper when $\operatorname{deg}(\beta)>1$.
    The presence of the identity in $x\bI$ is indeed used in previous arguments on universality of $\bbZ[x]$-modules for random matrices without symmetry constraints \cite{cheong2023distribution,van2025cokernel}. 
\end{remark}

\subsection{General reduction result}
Since the Haar measure on $\widehat{R}$ induces independent elements on the different factors, the pieces in the decomposition \eqref{eq:BoldAnt} are independent. 
It hence suffices to study their distributions separately. 
The key fact for this purpose will be that $R_{p,\beta}$ is a compact Noetherian local ring.
We present our arguments in this general setting and return to $R_{p,\beta}$ in \Cref{cor: ReductionBounded}.

All rings in the present paper are commutative and unital.
Recall that a ring is \emph{Noetherian} if all ideals are finitely generated. 
A \emph{local ring} is a ring $\sO$ with a unique maximal ideal $\mathfrak{m}\subsetneq\sO$.
Krull's intersection theorem yields that $\cap_{k>0} \mathfrak{m}^k = 0$ for any Noetherian local ring \cite[Corollary 10.19]{atiyah2018introduction}, so we can define a metric on $\sO$ by $d(x,y) \de 2^{-\nu(x-y)}$ with $\nu(r) \de \sup\{k\geq 0:r \in \mathfrak{m}^k   \}$. 
From here on, we fix a Noetherian local ring $\sO$. 
We further assume that $\sO$ is compact as a metric space, which is equivalent to the the residue field $\sO/\mathfrak{m}$ being finite.\footnote{
Indeed, note that $\mathfrak{m}^j/\mathfrak{m}^{j+1}$ is a finite-dimensional vector space over $\sO/\mathfrak{m}$ by Noetherianity and hence $\sO/\mathfrak{m}^j$ is finite if and only if $\sO/\mathfrak{m}$ is so. 
That $\sO/\mathfrak{m}^j$ is finite for all $j$ is equivalent to compactness of $\sO$, as may be verified directly from the definition of the metric on the local ring.\label{foot: Local}
} 
In particular, $(\sO,+)$ is then a compact Hausdorff topological group and hence admits a unique Haar probability measure. 
Denote $q \de \# \sO/\mathfrak{m}$. 
We use $\rank_q(\cdot)$ to refer to the rank over $\bbF_q \cong \sO/\mathfrak{m}$ of the reduction modulo $\mathfrak{m}$ of a matrix over $\sO$.

Now consider a random matrix $\bM \sim \operatorname{SymHaar}(\sO^{n\times n})$. 
Then, the reduction $\bM\bmod \mathfrak{m}$ is a uniform random symmetric matrix over $\bbF_q$. 
The number of symmetric matrices over a finite field with a given rank is classical, dating back to Carlitz \cite{carlitz1954representations} for odd $q$, and MacWilliams \cite{macwilliams1969orthogonal} for both odd and even $q$.  
It follows from their results (see \eg \cite[Eq.(20)]{fulman2015stein}) that for any fixed integer $k\geq 0$,  
\begin{align}
    \lim_{n\to \infty}\bbP\Bigl(\operatorname{rank}_{q}\bigl(\bM\bigr) = n-k \Bigr) = \frac{\prod_{i\geq 0}\bigl(1 - q^{-2i-1}\bigr)}{\prod_{i=1}^k(q^i -1)}, \label{eq:BlueClaw}
\end{align}
where the empty product yields unity if $k=0$. 
The goal in the remainder of this section is to understand the distribution of the cokernel conditional on the rank. 
This is achieved in \Cref{prop: Reduction} which yields a reduction to matrices of bounded dimensionality and is the core mathematical ingredient for the subsequent computations. 
First, we consider preliminaries in \Cref{lem: BlockDiagonalDecomposition,lem: InvarianceSymHaar}.  

\begin{lemma}\label{lem: BlockDiagonalDecomposition}
    Consider a matrix $\bM \in \sO^{n\times n}$ that is symmetric, $\bM = \bM^{\T}$. 
    Then, for every $k\leq n$, it holds that $\rank_q(\bM) = n-k$ if and only if there exists invertible $\bG \in \operatorname{GL}_n(\sO)$ such that 
    \begin{align}
        \bG^{\T}\bM\bG = \begin{pmatrix}
            \bQ & 0 \\ 
            0 & \bK
        \end{pmatrix}\label{eq:SmartOak}
    \end{align}
    for some symmetric $\bK \in \mathfrak{m}^{k\times k}$ and invertible symmetric $\bQ \in \operatorname{GL}_{n-k}(\sO)$.    
\end{lemma}
\begin{proof}
   A square matrix over $\sO$ is invertible if and only if its reduction modulo $\mathfrak{m}$ is so over $\bbF_q \cong \sO/\mathfrak{m}$. 
   (This follows from the adjoint formula for the inverse by using that the invertible elements of a local ring are those that are not in the maximal ideal to the determinant.
   Alternatively, one can use Nakayama's lemma.) 
   That $\rank_{q}(\bM) = n-k$ whenever \eqref{eq:SmartOak} is satisfied is hence immediate by taking the reduction modulo $\mathfrak{m}$ since $\bG$ and $\bQ$ then reduce to invertible matrices over $\bbF_q$. 

   Now suppose that $\rank_q(\bM) = n-k$. 
   To start, we then claim that there then exist invertible matrices $\bG_0 \in \operatorname{GL}_n(\bbF_q)$ and $\bQ_0 \in \operatorname{GL}_{n-k}(\bbF_q)$ over $\bbF_q$ with  
   \begin{align}
         \bG_0^{\T}\bM \bG_0  \equiv  
         \begin{pmatrix}
            \bQ_0 & 0\\ 
            0 & 0
        \end{pmatrix} \mod \mathfrak{m}.\label{eq:GoodCamel} 
    \end{align} 
    Indeed, let $e_1,\ldots,e_n \in \bbF_q^n$ be the standard basis and note that the assumption on the rank of $\bM$ ensures that there exists some invertible $\bG_0\in \operatorname{GL}_n(\bbF_q)$ that maps the final $k$ basis vectors into the kernel of $\bM \bmod \mathfrak{m}$.
    That is, with $\bM \bG_0 e_{n-i} = 0$ over $\bbF_q$ for all $i<k$. 
    Then also $e_{n-i}^{\T}\bG_0^{\T} \bM  =0$ by taking the transpose. 
    Thus, the final $k$ rows and columns of $\bG_0^{\T}\bM \bG_0$ are zero, and the remaining $(n-k)\times (n-k)$ block has to be invertible since $\bM_0$ has rank $n-k$. 

    Pick arbitrary lifts of $\bG_0$ and $\bQ_0$ to matrices over $\sO$, subject to the constraint that $\bQ_0$ remains symmetric. 
    Then, these lifts are invertible over $\sO$. 
    Moreover, \eqref{eq:GoodCamel} means that 
    \begin{align}
         \bG_0^{\T}\bM \bG_0  = 
         \begin{pmatrix}
            \bQ_0 & 0\\ 
            0 & 0
        \end{pmatrix} 
        + 
        \begin{pmatrix}
            \bD_{1,1} & \bD_{1,2}\\ 
            \bD_{1,2}^{\T} & \bD_{2,2}
        \end{pmatrix}\label{eq:ZombieGym} 
    \end{align}
    for certain matrices $\bD_{i,j}$ of appropriate dimension whose entries are in $\mathfrak{m}$ and where the diagonal matrices are symmetric, \ie $\bD_{i,i} = \bD_{i,i}^{\T}$.  
    Then,  with $\bI$ the identity matrix and $\bQ_1 \de \bQ_0 + \bD_{1,1}$,
    \begin{align}
        \bG_1^{\T}\bG_0^{\T}\bM \bG_0 \bG_{1} = \begin{pmatrix}
            \bQ_0 + \bD_{1,1} & 0 \\ 
            0 & \bD_{2,2} - \bD_{1,2}^{\T}\bQ_1^{-1} \bD_{1,2} 
        \end{pmatrix} 
        \quad 
        \textnormal{ with }
        \quad
        \bG_1 \de 
        \bI -
        \begin{pmatrix}
            0 & \bQ_1^{-1} \bD_{1,2}\\ 
            0 & 0 
        \end{pmatrix}.\label{eq:VagueDad}
    \end{align}
    Let $\bG = \bG_0 \bG_1$ and $\bQ \de \bQ_1$, and take $\bK =\bD_{2,2} - \bD_{1,2}^{\T}\bQ_1^{-1} \bD_{1,2}$ to conclude.  
\end{proof}
Let it be understood that the \emph{Haar measure on an ideal $\mathfrak{a} \subseteq\sO$} refers to the the unique probability measure that is preserved by additive translation with elements from $\mathfrak{a}$.
If $a_1,\ldots,a_r \in \mathfrak{a}$ are generators for the ideal, then this distribution corresponds to the law of $a_1H_1 + \cdots +a_k H_k$ with $H_i$ independent $\operatorname{Haar}(\sO)$-distributed elements.   
Define a probability distribution $\operatorname{SymHaar}(\mathfrak{a}^{n\times n})$ on symmetric matrices exactly like in \eqref{eq: DefSymHaar}.  

\begin{lemma}[Invariance property]\label{lem: InvarianceSymHaar}
     Suppose that $\bM \sim \operatorname{SymHaar}(\mathfrak{a}^{n\times n})$ and consider some deterministic invertible matrix $\bG \in \operatorname{GL}_n(\sO)$. 
     Then, the matrix $\bG^{\T} \bM \bG$ is again $\operatorname{SymHaar}(\mathfrak{a}^{n\times n})$-distributed.
\end{lemma}
\begin{proof}
    Suppose that $\bM \sim \operatorname{SymHaar}(\mathfrak{a}^{n\times n})$ and consider some arbitrary matrix $\bS \in \mathfrak{a}^{n\times n}$ that is symmetric $\bS = \bS^{\T}$. 
    Then,
    $
        \bG^{\T} \bM \bG + \bS = \bG^{\T} \bigl(\bM + \bS'\bigr) \bG$ with $\bS' \de (\bG^{\T})^{-1} \bS \bG^{{-1}}.    
    $
    Now, using that the Haar measure is preserved by additive translations on $\bM + \bS'$ yields that  $\bG^{\T} \bM \bG + \bS$ has the same distribution as $\bG^{\T} \bM \bG$. 
    Thus, the law of $\bG^{\T} \bM \bG$ is also preserved by translation.
    Recall that invariance by translation characterizes the Haar measure to conclude. 
\end{proof}

\begin{proposition}\label{prop: Reduction}
    Suppose that $\bM \sim \operatorname{SymHaar}(\sO^{n\times n})$. 
    Fix some $k\leq n$. 
    Then, we have an equality in distribution of random $\sO$-modules: 
    \begin{align}
        \bbP\Bigl(\coker\bigl(\bM \bigr)  \in * \mid \rank_{q}\bigl(\bM \bigr) = n-k\Bigr) = \bbP\Bigl(\coker\bigl(\bK\bigr) \in * \Bigr),
    \end{align}
    where $\bK  \sim \operatorname{SymHaar}(\mathfrak{m}^{k\times k})$ is a symmetric $k\times k$ matrix with entries Haar distributed on $\mathfrak{m}$. 
\end{proposition} 
\begin{proof}
    Let $\bM$ be a $\operatorname{SymHaar}(\sO^{n\times n})$ distributed matrix conditioned to have rank $n-k$. 
    Then, \Cref{lem: BlockDiagonalDecomposition} ensures there exist random invertible $\bG \in \operatorname{GL}_n(\sO)$ and invertible symmetric $\bQ \in \operatorname{GL}_{n-k}(\sO)$ as well as a random symmetric matrix $\bK \in \mathfrak{m}^{k\times k}$ such that $\bG^{\T}\bM\bG =   \operatorname{diag}(\bQ,\bK) $.

    We claim that it can here be assumed without loss of generality that $\bK\sim \operatorname{SymHaar}(\mathfrak{m}^{k\times k})$.   
    To see this, note that $\bM$ has the same distribution as $\bM + \bDelta$ with $\bDelta \sim \operatorname{SymHaar}(\mathfrak{m}^{n\times n})$ an independent random matrix. 
    \Cref{lem: InvarianceSymHaar} further implies that $\bD\de  \bG^{\T}\bDelta \bG$ is again $\operatorname{SymHaar}(\mathfrak{m}^{n\times n})$ distributed and independent from $\bM, \bK, \bG$ and $\bQ$.
    The matrix $\bG^{\T} (\bM + \Delta) \bG = \bG^{\T} \bM \bG + \bD$ may not be in block diagonal form anymore, but off-diagonal blocks can be cleared exactly as in \eqref{eq:ZombieGym}--\eqref{eq:VagueDad} in the proof of \Cref{lem: BlockDiagonalDecomposition}. 
    That is, writing $\bD$ in block diagonal form as in \eqref{eq:ZombieGym}, it holds with $\bQ_{\Delta} \de \bQ + \bD_{1,1}$ that  
    \begin{align}
        \bG_\Delta^{\T}\bG^{\T}(\bM + \Delta) \bG \bG_{\Delta} = \begin{pmatrix}
            \bQ_{\Delta} & 0 \\ 
            0 &\bK + \bD_{2,2} - \bD_{1,2}^{\T}\bQ_\Delta^{-1} \bD_{1,2} 
        \end{pmatrix} 
        \quad 
        \textnormal{ with }
        \quad
        \bG_{\Delta} \de 
        \bI -
        \begin{pmatrix}
            0 & \bQ^{-1}_\Delta \bD_{1,2}\\ 
            0 & 0 
        \end{pmatrix}
    \end{align}
    Here, note that $\bD_{2,2}$ is independent from $\bD_{1,2}^{\T}\bQ_{\Delta} \bD_{1,2}$ and $\bK$ due to the independence of $\bD$ from $\bQ$ and $\bK$, as well as the independence of the blocks in $\bD$ that follow from the matrix being $\operatorname{SymHaar}(\mathfrak{m}^{n\times n})$ distributed. 
    Consequently, the translation invariance of $\operatorname{SymHaar}(\mathfrak{m}^{k\times k})$ applied to $\bD_{2,2}$ ensures that the lower-right block is again Haar distributed: 
    \begin{align}
        \bK + \bD_{2,2} - \bD_{1,2}^{\T}\bQ_\Delta^{-1} \bD_{1,2} \sim \operatorname{SymHaar}(\mathfrak{m}^{k\times k}).\label{eq:SafeYawn}
    \end{align}
    Thus, we can indeed assume that $\bG^{\T} \bM \bG=  \operatorname{diag}(\bQ,\bK) $ with  $\bK \sim \operatorname{SymHaar}(\mathfrak{m}^{k\times k})$.
    (Indeed, replace $\bM$ by $\bM + \Delta$, replace $\bK$ by $\bK + \bD_{2,2} - \bD_{1,2}^{\T}\bQ_\Delta^{-1} \bD_{1,2}$, and replace $\bG$ and $\bQ$ by $\bG\bG_{\Delta}$ and $\bQ_{\Delta}$, respectively.) 

    It follows directly from the definition of the cokernel \eqref{eq: Def_Coker} that $\coker(\bU \bM \bV) \cong \coker(\bM)$ for every $\bU,\bV \in \operatorname{GL}_{n}(\sO)$. 
    Moreover, also directly from the definition, the cokernel of a block diagonal matrix is the direct sum of the cokernels of its diagonal blocks, and an invertible matrix always has trivial cokernel. 
    Hence, 
    \begin{align}
        \coker(\bM ) \cong \coker(\bG^{\T} \bM \bG ) \cong \coker(\bQ) \oplus \coker(\bK) \cong \coker(\bK). 
    \end{align}
    This concludes the proof.  
\end{proof}

\begin{corollary}\label{cor: ReductionBounded}
    Suppose that $\bM_{p,\beta} \sim \operatorname{SymHaar}(R_{p,\beta}^{n\times n})$. 
    Let it be understood that $\rank_q(\cdot)$ refers to the rank over $\bbF_q \cong \bbZ[x]/(p\bbZ[x] + \beta(x)\bbZ[x])$. 
    Then, for every $k\leq n$,  as random $R_{p,\beta}$ modules, 
    \begin{align}
        \bbP\Bigl(\coker\bigl(\bM_{p,\beta},\bigr) \in  *  \mid \rank_{q}\bigl(\bM_{p,\beta}\bigr) = n-k\Bigr) = \bbP\Bigl(\coker\bigl(\bK  \bigr) \in  * \Bigr),
    \end{align}
    where $\bK$ is a symmetric $k\times k$ matrix with entries Haar distributed on $pR_{p,\beta} + \beta(x) R_{p,\beta}$. 
    That is, 
    \begin{align}
        \bK  \sim  p\operatorname{SymHaar}(R_{p,\beta}^{k\times k})  + \beta(x) \operatorname{SymHaar}(R_{p,\beta}^{k\times k}).   
    \end{align}
\end{corollary}
\begin{proof}
    This is immediate from \Cref{prop: Reduction} since $R_{p,\beta}$ is a compact Noetherian local ring with maximal ideal $\widehat{\mathfrak{m}} = pR_{p,\beta} + \beta(x) R_{p,\beta}$. 
    Indeed, the $\mathfrak{m}$-adic completion of a ring at a maximal ideal $\mathfrak{m}$ is always a local ring \cite[Proposition 10.16]{atiyah2018introduction}, it is Noetherian when the original ring is so \cite[Theorem 10.26]{atiyah2018introduction}, and it is compact if and only if the residue field modulo $\mathfrak{m}$ is finite (recall footnote \ref{foot: Local}). 
\end{proof}

\subsection{Conditions based on the walk matrix}\label{sec: ConditionWalk}

The following definition provides a profinite analogue for the rephrasing of the condition in \Cref{prop: ModuleRephrasingWalkDet}: 

\begin{definition}\label{def: ConditionW}
    Fix a prime $p$. 
    Then, a $\widehat{R}$-module $\sM$ is said to satisfy \emph{condition $\cW_p$} if one of the following two events is satisfied:
    \begin{enumerate}[leftmargin = 2.8em]
        \item[(W1)] It holds that $\sM\otimes_{\widehat{R}} R_{p,\beta} \cong 0$ for every monic $\beta(x) \in \bbZ[x]$ with irreducible reduction in $\bbF_p[x]$.
        \item[(W2)] Or, there exists $a\in \bbZ$ such that $\sM \otimes_{\widehat{R}} R_{p,x-a} \cong \bbF_{p}[x]/(x-a)\bbF_{p}[x]$ 
        and $\sM\otimes_{\widehat{R}} R_{p,\beta} \cong 0$ for every $\beta \not\equiv x-a \bmod p$.     
    \end{enumerate}
\end{definition}

To model \Cref{conj: WalkMatrixSquareFree}, we study the frequency of condition $\cW_p$ for the module $\coker(\bM-x\bI,\zeta)$ with $\zeta \sim \operatorname{Haar}(\widehat{R}^n)$ a random vector that is independent of the random matrix $\bM$. 
Exactly as in \eqref{eq:WeepyJar}--\eqref{eq:BoldAnt}, the random modules $\coker(\bM-x\bI, \zeta) \otimes_{\widehat{R}} R_{p,\beta}$ for varying $(p,\beta)$ are independent with the same distribution as $\coker(\bM_{p,\beta}, \zeta_{p,\beta})$ for $\bM_{p,\beta}\sim \operatorname{SymHaar}(R_{p,\beta}^{n\times n})$ and $\zeta_{p,\beta} \sim \operatorname{Haar}(R_{p,\beta}^n)$. 
It hence suffices to study the distribution of these pieces individually. 
\Cref{cor: ReductionBounded} reduces the latter problem to a finite direct computation, whose details we next provide in \Cref{lem: TrivialCoker_pbeta,lem: 1D-cokernel}:

\begin{lemma}\label{lem: TrivialCoker_pbeta}
    Let $q \de p^{\operatorname{deg}(\beta)}$. 
    Then, for every $n\geq 1$, 
    \begin{align}
       \Bigl\lvert\, \bbP\bigl( \coker(\bM_{p,\beta}, \zeta_{p,\beta}) \cong 0\bigr) - \Bigl( 1 - \frac{1}{q^2}\Bigr)\prod_{i\geq 1}\Bigl(1 - \frac{1}{q^{2i+1}} \Bigr) \,  \Bigr\rvert \leq  \frac{6}{q^n}. \label{eq:TediousPop}
    \end{align}
\end{lemma}
\begin{proof}
    Nakayama's lemma \cite[Proposition 2.6]{atiyah2018introduction} applied over $R_{p,\beta}$ yields that $\coker(\bM_{p,\beta}, \zeta_{p,\beta}) \cong 0$ if and only if $[\bM_{p,\beta}, \zeta_{p,\beta}]$ has full rank over $\bbF_q \de \bbF_{p}[x]/\beta(x) \bbF_p[x]$. 
    The rank of that rectangular matrix is at most one higher than that of $\bM_{p,\beta}$. 
    Hence, by the law of total probability,  
    \begin{align}
        \bbP\bigl( \coker(\bM_{p,\beta}, \zeta_{p,\beta}) \cong 0\bigr){}&{} =  \bbP\bigl(\operatorname{rank}_q(\bM_{p,\beta}) = n\bigr)\label{eq:ZippySong}\\ 
        &+ \bbP\bigl(\operatorname{rank}_q(\bM_{p,\beta}) = n-1 \bigr) \bbP(\operatorname{rank}_q(\bM_{p,\beta},\zeta_{p,\beta}) = n\mid \operatorname{rank}_q(\bM_{p,\beta}) = n-1 ). \nonumber
    \end{align}
    Fulman and Goldstein \cite[Theorem 4.1]{fulman2015stein} gave bounds on the total variation distance between the probability distribution of the rank of a uniform random symmetric $n\times n$ matrix over $\bbF_q$ and the limiting law \eqref{eq:BlueClaw}. 
    In particular, their bounds imply that for every $k\leq n$, 
    \begin{align}
        \textstyle \bigl\lvert \bbP\bigl(\operatorname{rank}_q(\bM_{p,\beta}) = n-k\bigr) -\prod_{i=1}^k(q^i -1)^{-1} \prod_{i\geq 0}\bigl(1 - q^{-2i-1}\bigr) \bigr\rvert \leq 3/q^{n}. \label{eq:NoisyHut}  
    \end{align}
    Conditional on the event $\operatorname{rank}_q(\bM_{p,\beta}) = n-1 $ it occurs that $\operatorname{rank}_q(\bM_{p,\beta},\zeta_{p,\beta}) = n$ if and only if $\zeta_{p,\beta}$ reduces to a nonzero element in the one-dimensional $\bbF_q$ vector space $\coker(\bM_{p,\beta})\otimes_{R_{p,\beta}} \bbF_q$. 
    Thus, since $\zeta_{p,\beta}$ yields a uniform random element of that one-dimensional vector space,  
    \begin{align}
        \bbP(\operatorname{rank}_q(\bM_{p,\beta},\zeta_{p,\beta}) = n\mid \operatorname{rank}_q(\bM_{p,\beta}) = n-1 )  = 1- 1/q.\label{eq:TrustyOak} 
    \end{align}
    Substitute \eqref{eq:NoisyHut} and \eqref{eq:TrustyOak} in \eqref{eq:ZippySong} and simplify by using that $1 + (q-1)^{-1}(1-q^{-1})= 1 + q^{-1}$ to find \eqref{eq:TediousPop}. 
    Here, the product in \eqref{eq:TediousPop} starts at $i=1$ because the factor with $i=0$ from \eqref{eq:NoisyHut} was rewritten using that $(1+q^{-1})(1-q^{-1}) = 1-q^{-2}$.  
\end{proof}

\begin{lemma}\label{lem: 1D-cokernel}
    Let $q \de p^{\deg(\beta)}$.
    Then, every $n\geq 2$, 
    \begin{align}
          \biggl\lvert\, \bbP\biggl(\coker\bigl(\bM_{p,\beta},\zeta_{p,\beta}\bigr)    \cong \frac{\bbF_{p}[x]}{\beta(x)\bbF_{p}[x]}\biggr) - \frac{1}{q^2}\Bigl( 1 - \frac{1}{q^2}\Bigr)^2 \prod_{i\geq 1}\Bigl(1 - \frac{1}{q^{2i+1}} \Bigr)   \,  \biggr\rvert \leq \frac{6}{q^n}. \label{eq:SafeKee}
    \end{align}
\end{lemma}
\begin{proof}
    Denote $\bbF_q \de \bbF_{p}[x]/\beta(x)\bbF_{p}[x]$. 
    Note that $\coker\bigl(\bM_{p,\beta},\zeta_{p,\beta}\bigr)  \cong \bbF_q$ as a $R_{p,\beta}$-module implies that $\coker\bigl(\bM_{p,\beta},\zeta_{p,\beta}\bigr)\otimes_{R_{p,\beta}} \bbF_q \cong \bbF_{q}$, and hence necessitates that the matrix $[\bM_{p,\beta}, \zeta_{p,\beta}]$ reduces to a matrix with rank $n-1$ over $\bbF_q$. 
    Hence, by the law of total probability, 
    \begin{align}
        \bbP\bigl(\coker\bigl(\bM_{p,\beta},{}&{}\zeta_{p,\beta}\bigr)    \cong  \bbF_q\bigr) \label{eq:BlueKey}\\ 
        & = \bbP\bigl(\coker\bigl(\bM_{p,\beta},\zeta_{p,\beta}\bigr)    \cong\bbF_q\ \Big\vert \   \rank_q(\bM_{p,\beta}) = n-1\bigr) \bbP\bigl( \rank_q(\bM_{p,\beta}) = n-1\bigr) \nonumber\\ 
        & \ \ +  \bbP\bigl(\coker\bigl(\bM_{p,\beta},\zeta_{p,\beta}\bigr)    \cong\bbF_q \ \Big\vert \   \rank_q(\bM_{p,\beta}) = n-2\bigr) \bbP\bigl( \rank_q(\bM_{p,\beta}) = n-2\bigr).\nonumber
    \end{align}
    
    We start with the probability conditional on rank $n-1$. 
    Then, \Cref{cor: ReductionBounded} yields that $\coker(\bM_{p,\beta})$ has the same distribution as $\coker(k_1)$ with $k_1\sim p\operatorname{Haar}(R_{p,\beta}) + \beta(x)\operatorname{Haar}(R_{p,\beta})$. 
    Consequently, $\coker(\bM_{p,\beta}, \zeta_{p,\beta})$ then has the same distribution as $\coker(k_1,z)$ with $z\sim \operatorname{Haar}(R_{p,\beta})$. 
    If $z$ has nonzero reduction in $\bbF_q$ then $\coker(z)\cong 0$ and hence also $\coker(k_1,z)\cong 0$. 
    Hence, since $z$ reduces to zero with probability $1/q$, it holds with $k_1,k_2 \sim  p\operatorname{Haar}(R_{p,\beta}) + \beta(x)\operatorname{Haar}(R_{p,\beta})$ independent that  
    \begin{equation}
         \bbP\bigl(\coker\bigl(\bM_{p,\beta(x)},\zeta_{p,\beta}\bigr)    \cong \bbF_q  \mid     \rank_q(\bM_{p,\beta}) = n-1\bigr) = q^{-1} \bbP\bigl(\coker(k_1,k_2)  \cong \bbF_q \bigr).\label{eq:FunBox}
    \end{equation}
    Nakayama's lemma \cite[Proposition 2.8]{atiyah2018introduction} implies that two elements of $\widehat{\mathfrak{m}} =  pR_{p,\beta} + \beta(x) R_{p,\beta}$ generate this ideal if and only if their reduction to $\widehat{\mathfrak{m}}/\widehat{\mathfrak{m}}^2$ generate the latter as a vector space over $\bbF_q$. 
    Here, $\widehat{\mathfrak{m}}/\widehat{\mathfrak{m}}^2 \cong \bbF_q^2$ as a vector space since $p$ and $\beta(x)$ yield a basis.
    Consequently, using that $R_{p,\beta}/\widehat{\mathfrak{m}} \cong \bbF_q$ and that the probability that two random vectors generate $\bbF_q^2$ is exactly $(1 - 1/q^{2})(1-1/q)$, 
    \begin{align}
         \bbP\bigl(\coker\bigl(\bM_{p,\beta(x)},\zeta_{p,\beta}\bigr)    \cong \bbF_q  \mid   \rank_q(\bM_{p,\beta}) = n-1\bigr) = q^{-1} \bigl(1 - q^{-2}\bigr) \bigl(1-q^{-1} \bigr).\label{eq:DizzyCrow} 
    \end{align}

    The probability conditional on rank $n-2$ proceeds with a similar strategy.
    On that event, \Cref{cor: ReductionBounded} yields that $\coker(\bM_{p,\beta})$ has the same distribution as $\coker(\bK)$ with $\bK$ a symmetric $2\times 2$ matrix with entries Haar distributed on $ pR_{p,\beta} + \beta(x)R_{p,\beta}$. 
    Then, $\coker(\bM_{p,\beta},\zeta_{p,\beta})$ has the same distribution as $\coker(\bK, Z)$ with $Z\sim \operatorname{Haar}(R_{p,\beta}^2)$. 
    If $Z$ reduces to zero in $\bbF_q^2$, then $\coker(\bK, Z)\otimes_{R_{p,\beta}}\bbF_q$ is a $2$-dimensional vector space so that $\coker(\bK,Z)$ could not be isomorphic to $\bbF_q$ as a $R_{p,\beta}$-module.
    Hence, since $Z$ reduces to zero with probability $1/q^{2}$,  
    \begin{align}
         \bbP\bigl(\coker\bigl(\bM_{p,\beta},\zeta_{p,\beta}{}&{}\bigr)    \cong\bbF_q  \mid    \rank_q(\bM_{p,\beta}) = n-2\bigr)\label{eq:PaleTea} \\ 
         &= \bigl( 1 - q^{-2}\bigr) \bbP\bigl( \coker(\bK , Z) \cong \bbF_q  \mid Z\not\equiv 0 \bmod p R_{p,\beta} + \beta(x) R_{p,\beta}  \bigr).\nonumber 
    \end{align}
   That $Z$ has nonzero reduction is equivalent to the existence of a matrix $\bG \in \operatorname{GL}_2(R_{p,\beta})$ such that $\bG Z = (1,0)^{\T}$.     
    The matrix $\bG$ only depends on $Z$, so the invariance from \Cref{lem: InvarianceSymHaar} implies that $\tilde{\bK}=\bQ \bK \bQ^{\T}$ has the same distribution as $\bK$. 
    Further,  
    \begin{equation}
        \coker(\bK, Z) \cong \coker(\bQ \bK \bQ^{\T}, \bQ Z)  
        = \coker\begin{pmatrix}
            \tilde{\bK}_{1,1} & \tilde{\bK}_{1,2} & 1 \\ 
            \tilde{\bK}_{1,2} & \tilde{\bK}_{2,2} & 0
        \end{pmatrix} 
        \cong  \coker(\tilde{\bK}_{1,2},\, \tilde{\bK}_{2,2}).  
    \end{equation}
    Note that $\tilde{\bK}_{1,2}\, , \,  \tilde{\bK}_{2,2} \sim p\operatorname{Haar}(R_{p,\beta}) + \beta(x)\operatorname{Haar}(R_{p,\beta})$. 
    Consequently, by the arguments in \eqref{eq:FunBox}--\eqref{eq:DizzyCrow}, it holds that $\bbP(\coker(\tilde{\bK}_{1,2}, \tilde{\bK}_{2,2}) \cong \bbF_q) = (1-1/q^2)(1-1/q)$. 
    Hence, using \eqref{eq:PaleTea}, 
    \begin{align}
        \bbP\bigl(\coker(\bM_{p,\beta},\zeta_{p,\beta} )    \cong\bbF_q  \mid   \rank_q(\bM_{p,\beta}) = n-2\bigr) = \bigl( 1 - q^{-2}\bigr)^2 \bigl( 1 - q^{-1}\bigr).\label{eq:ZenCity}  
    \end{align}

    Substitution of \eqref{eq:DizzyCrow} and \eqref{eq:ZenCity} in \eqref{eq:BlueKey} now yields that 
    \begin{align}
        \bbP\bigl(\coker\bigl(\bM_{p,\beta}, \zeta_{p,\beta}\bigr)    \cong  \bbF_q\bigr) ={}&{}  q^{-1} \bigl(1 - q^{-2}\bigr) \bigl(1-q^{-1} \bigr) \bbP\bigl( \rank_q(\bM_{p,\beta}) = n-1\bigr) \label{eq:OccultDew}  \\ 
        & \quad +  \bigl( 1 - q^{-2} \bigr)^2 \bigl( 1 - q^{-1}\bigr)\bbP\bigl( \rank_q(\bM_{p,\beta}) = n-2\bigr). \nonumber 
    \end{align}
    By \cite[Theorem 4.1]{fulman2015stein}, we can replace $\bbP( \rank_q(\bM_{p,\beta}) = n-k)$ by its limiting value up to an error $\leq 3/q^n$; recall \eqref{eq:NoisyHut}. 
    The desired result \eqref{eq:SafeKee} now follows after simplifying the prefactors, since  
    \begin{align}
         \frac{(1 - q^{-2}) (1-q^{-1})}{q (q-1)}  + \frac{(1-q^{-2})^2 (1-q^{-1})}{(q-1)(q^2-1)} = \frac{(1 - q^{-2})^2}{q^{2}(1-q^{-1})}.
    \end{align}
    Here, the product in \eqref{eq:SafeKee} starts at $i=1$ because $(1-q^{-1})^{-1}$ cancels the factor for $i=0$ in \eqref{eq:NoisyHut}.  
\end{proof}
It now remains to combine \Cref{lem: TrivialCoker_pbeta,lem: 1D-cokernel} which involves taking the product of the probabilities.
To simplify the latter we will use the following \Cref{lem: Zeta}:  
\begin{lemma}\label{lem: Zeta} 
    Fix a prime $p$. Then, for every real $s>1$, 
    \begin{align}
        \prod_{\beta(x) }\Bigl(1 - \frac{1}{p^{\operatorname{deg(\beta)}s}  } \Bigr)  =    1 - \frac{1}{p^{s-1}}.  
    \end{align}
    Here, the product runs over all monic irreducible $\beta(x) \in \bbF_p[x]$. 
\end{lemma}
\begin{proof}
    This follows from the Euler product formula for the zeta function of the ring $\bbF_p[x]$. 
    That is, by evaluating $\sum_{\textnormal{monic }f\in \bbF_{p}[x]} p^{-\deg(f)s}$ in two ways; see \eg \cite[Chapter 2, Equations (1)\&(2)]{rosen2013number}.  
\end{proof}

Combining \Cref{lem: TrivialCoker_pbeta,lem: 1D-cokernel,lem: Zeta} now yields our main result: the exact limiting probability that condition $\cW_p$ is satisfied in the profinite model; recall \Cref{def: ConditionW}. 
In particular, the following \Cref{thm: HaarMainWalk} justifies \Cref{conj: WalkMatrixSquareFree}.
\begin{theorem}\label{thm: HaarMainWalk}
    Let $\bM\sim \operatorname{SymHaar}(\widehat{R}^{n\times n})$ and $\zeta \sim \operatorname{Haar}(\widehat{R}^n)$. 
    Then, for every fixed prime $p$, 
    \begin{align}
       \lim_{n\to \infty}\bbP\Bigl(\coker(\bM-x\bI, \zeta) \textnormal{ satisfies condition }\cW_p \Bigr) = \Bigl(1 - \frac{1}{p^2} - \frac{1}{p^3} + \frac{1}{p^4}\Bigr) \prod_{i=1}^\infty \Bigl(1 - \frac{1}{p^{2i}}\Bigr).\label{eq:SlowWisp} 
    \end{align}
    Moreover, for any set of primes $\sP$, 
    \begin{align}
      \lim_{n\to \infty}\bbP\Bigl(\forall   p\in \sP : \coker(\bM-x\bI, \zeta) \textnormal{ satisfies  }\cW_p  \Bigr)  = \prod_{p\in \sP}  \Bigl(1 - \frac{1}{p^2} - \frac{1}{p^3} + \frac{1}{p^4}\Bigr) \prod_{i=1}^\infty \Bigl(1 - \frac{1}{p^{2i}}\Bigr).\label{eq:CheekyRock}
    \end{align}
\end{theorem}
\begin{proof}
     Consider the following abbreviations for any prime $p$, polynomial $\beta(x)$, and integer $a\in \bbZ$:   
     \begin{equation}
        P_n(p,\beta) \de \bbP\Bigl(\coker(\bM_{p,\beta}, \zeta_{p,\beta}) \cong 0 \Bigr),  \ \  Q_n(p,a) \de \bbP\Bigl(\coker(\bM_{p,x-a}, \zeta_{p,x-a}) \cong \frac{\bbF_p[x]}{(x-a)\bbF_{p}[x]} \Bigr).\nonumber 
     \end{equation}
    Note that the two events (W1) and (W2) in \Cref{def: ConditionW} are mutually exclusive. 
    Moreover, the event (W2) can be further subdivided in $p$ disjoint events corresponding to the options for $a\bmod p$. 
    Hence, using the law of total probability as well as the independence and law of the pieces associated with varying $\beta(x)$ that was remarked upon preceding \Cref{lem: TrivialCoker_pbeta}, 
    \begin{equation}
         \bbP\Bigl(\coker(\bM-x\bI, \zeta) \textnormal{ satisfies condition }\cW_p \Bigr) = \prod_{\beta(x)} P_n(p,\beta) + \sum_{a=0}^{p-1} Q_n(p,a)\negsp\negsp \prod_{\beta(x)\not\equiv x-a}\negsp\negsp P_n(p,\beta). \label{eq:AncientVine}
    \end{equation}
    Further, again using the independence, 
    \begin{equation}
        \bbP\Bigl(\forall   p\in \sP : \coker(\bM-x\bI, \zeta) \textnormal{ satisfies }\cW_p \Bigr)  = \prod_{p\in \sP}\bbP\Bigl(\coker(\bM-x\bI, \zeta) \textnormal{ satisfies }\cW_p \Bigr).\label{eq:NiftyVine}
    \end{equation}
    
    We start by formally computing the right-hand side of  \eqref{eq:AncientVine} for $n=\infty$.  
    \Cref{lem: TrivialCoker_pbeta,lem: 1D-cokernel} ensure that $P_\infty(p,\beta) \de \lim_{n\to \infty}P_n(p,\beta)$ and $Q_\infty(p,a)\de \lim_{n\to \infty}Q_n(p,\beta)$ exist.  
    Substituting the values for $P_\infty(p,\beta)$ from  \Cref{lem: TrivialCoker_pbeta} and subsequently using \Cref{lem: Zeta} yields that  
   \begin{equation}
       \prod_{\beta(x)} P_\infty(p,\beta) 
        = \prod_{\beta(x) }\biggl( \Bigl(1-\frac{1}{p^{2\operatorname{deg(\beta)}}}\Bigr) \prod_{i\geq 1} \Bigl(1- \frac{1}{p^{(2i+1)\operatorname{deg(\beta)}}}\Bigr) \biggr)
        = \Bigl(1 - \frac{1}{p} \Bigr)\prod_{i\geq 1}  \Bigl(1-\frac{1}{p^{2i}} \Bigr). \label{eq:FastFog} 
   \end{equation}
   Comparing \Cref{lem: TrivialCoker_pbeta,lem: 1D-cokernel} shows that $Q_\infty(p,a) = p^{-2}(1-p^{-2}) P_\infty(p,x-a)$. 
   In particular,
   \begin{equation}
       \sum_{a=0}^{p-1} Q_\infty(p,a) \negsp\prod_{\beta(x)\not\equiv x-a}\negsp P_\infty(p,\beta) = \frac{1}{p}\Bigl(1-\frac{1}{p^{2}} \Bigr)\prod_{\beta(x)} P_\infty(p,\beta)
   \end{equation}
   Consequently, using \eqref{eq:FastFog} and that $(1 + p^{-1}(1-p^{-2}) ) (1 - p^{-1}) = 1 - p^{-2}  -p^{-3} +p^{-4}$,   
    \begin{equation}
        \prod_{\beta(x)} P_\infty(p,\beta) + \sum_{a=0}^{p-1} Q_\infty(p,a)\negsp \prod_{\beta(x)\not\equiv x-a}\negsp\negsp P_\infty(p,\beta) =\Bigl(1 - \frac{1}{p^2} - \frac{1}{p^3} + \frac{1}{p^4}\Bigr) \prod_{i=1}^\infty \Bigl(1 - \frac{1}{p^{2i}}\Bigr). \label{eq:KnownFace} 
    \end{equation}
    It remains to justify that the limit and product may be exchanged in \eqref{eq:AncientVine} and \eqref{eq:NiftyVine}.
    This follows from the error bounds in \Cref{lem: TrivialCoker_pbeta,lem: 1D-cokernel} by a direct computation, whose details we next provide.

    Define scalars $\mathfrak{p}_{n}(p,\beta), \mathfrak{q}_n(p,a) \in \bbR$ by 
    $
         P_n(p,\beta) = (1 + \mathfrak{p}_n(p,\beta)) P_\infty(p,\beta)$ 
    and 
    $Q_n(p,a) = (1+\mathfrak{q}_n(p,a)) Q_\infty(p,a). 
    $
    The explicit value for $P_\infty(p,\beta)$ in  \Cref{lem: TrivialCoker_pbeta} is an increasing function of $p$ and the degree of $\beta$.
    In particular, $P_\infty(p,\beta) \geq P_\infty(p,x) \geq P_\infty(2,x) \geq c_1$ for some absolute constant $c_1>0$.
    Hence, the error bounds from \Cref{lem: TrivialCoker_pbeta} imply that $\lvert \mathfrak{p}_n(p,\beta) \rvert \leq c_2 p^{-n\operatorname{deg}(\beta)}$ with $c_2 = 6/c_1$.
    It similarly follows from \Cref{lem: 1D-cokernel} that there exists $c_3 >0$ with $\lvert \mathfrak{q}_n(p,a) \rvert \leq c_3 p^{-(n-2)}$. 
    Let $c_4>0$ be a sufficiently large constant so that $c_2 p^{-n} \leq p^{-(n-c_4)}$ and $c_3 p^{-(n-2)} \leq p^{-(n-c_4)}$.

    There are at most $p^d$ monic irreducible polynomials $\beta \in \bbF_p[x]$ with degree $d$. 
    Hence,  for $n>c_4$, 
    \begin{align}
         &\textstyle \prod_{d\geq 1}\bigl(1 - p^{-d (n -c_4) }\bigr)^{p^d} \leq \prod_{\beta} \bigl(1+\mathfrak{p}_n(p,\beta)\bigr)  \leq  \prod_{d\geq 1}\bigl(1 +  p^{-d (n-c_4) }\bigr)^{p^d},\label{eq:HugeLog} \\ 
         &\textstyle \prod_{d\geq 1} (1-  p^{-d(n-c_4)})^{p^d} \leq  \bigl(1+\mathfrak{q}(p,a)\bigr) \prod_{\beta \neq x-a} \bigl(1+\mathfrak{p}_n(p,\beta)\bigr)\leq    \prod_{d\geq 1} (1+  p^{-d(n-c_4)})^{p^d}. \label{eq:SpikyPug} 
    \end{align}
    Using that $1+x \leq \exp(x)$ for $x\geq 0$, we have $\prod_{d\geq 1} (1+  p^{-d(n-c_4)})^{p^d} \leq  \exp(\sum_{d\geq 1} p^{-d(n-c_4-1)})$.
    Summing the series and using a Taylor approximation, there hence exists $c_5 >0$ with $\prod_{d\geq 1} (1+  p^{-d(n-c_4)})^{p^d} \leq 1 + c_5 p^{-(n-c_4 -1)}$ for large $n$. 
    For the lower bound, note that a Taylor approximation yields $c_6>0$ with $1/(1-p^{-d(n-c_4)}) \leq 1 + c_6p^{-d(n -c_4)}$. 
    Absorbing the constant in the exponent allows upper bounding $1/\prod_{d}(1-p^{-d(n-c_4)})^{p^d}$ as above, which yields a lower bound on $\prod_{d}(1-p^{-d(n-c_4)})^{p^d}$. 
    It follows in this fashion that there exists an absolute constant $c_7 >0$ such that for all large $n$,  
    \begin{align}
        \textstyle \prod_{d\geq 1} \bigl(1+p^{-(n-c_4)}\bigr)^{p^d} \leq  1+ p^{- (n-c_7)}\  \textnormal{ and }\   \prod_{d\geq 1} \bigl(1-p^{-d(n-c_4)}\bigr)^{p^d} \geq  1- p^{-(n-c_7)}.\label{eq:SnappyDen}
    \end{align}
    
    Recall that  
    $
         P_n(p,\beta) = (1 + \mathfrak{p}_n(p,\beta)) P_\infty(p,\beta)$ 
    and 
    $Q_n(p,a) = (1+\mathfrak{q}_n(p,a)) Q_\infty(p,a) 
    $.
    It consequently follows from \eqref{eq:HugeLog}--\eqref{eq:SnappyDen} that for every set of primes $\sP$, 
    \begin{align}
        \prod_{p\in \sP} \Bigl(1 -  \frac{1}{p^{n-c_7}}\Bigr)  \leq \prod_{p\in \sP }\frac{P_n(p,\beta) + \sum_{a=0}^{p-1} Q_n(p,a) \prod_{\beta \neq x-a} P_n(p,\beta)}{P_\infty(p,\beta) + \sum_{a=0}^{p-1} Q_\infty(p,a) \prod_{\beta \neq x-a} P_\infty(p,\beta)} \leq \prod_{p\in \sP} \Bigl(1 +  \frac{1}{p^{n-c_7}}\Bigr). \label{eq:MadForce} 
    \end{align}
    Expanding the product in the lower bound \eqref{eq:MadForce} to all primes $p$ yields $1/\zeta(n -c_7)$ with $\zeta$ the Riemann zeta function, while the upper bound yields $\zeta(n-c_7)/\zeta(2(n-c_7))$.   
    Use that $\zeta(x) \to 1$ as $x\to +\infty$ to conclude from \eqref{eq:KnownFace} that \eqref{eq:CheekyRock} holds true, and hence also \eqref{eq:SlowWisp} as the case $\sP = \{p \}$. 
\end{proof}

\subsection{Conditions based on the discriminant}\label{sec: ResultsDiscriminant}

The following \Cref{def: ConditionD1,def: ConditionD2} give profinite analogues for the conditions in \Cref{prop: ModuleRephrasingDiscriminantCondition,prop: ModuleRephrasingDiscriminantConditionTwo}. 
Both definitions also allow $p=2$.   
\begin{definition}\label{def: ConditionD1}
    Fix a prime $p$.
    Then, a $\widehat{R}$-module $\sM$ is said to satisfy \emph{condition $\cD_p^{1}$} if it holds for every monic $\beta(x) \in \bbZ[x]$ with irreducible reduction in $\bbF_p[x]$ that $\widehat{\sM}_{p,\beta} \de \sM\otimes_{\widehat{R}} R_{p,\beta}$ satisfies
        \begin{equation}
              \widehat{\sM}_{p,\beta}/p\widehat{\sM}_{p,\beta} \cong 0 \qquad \textnormal{or} \qquad  \widehat{\sM}_{p,\beta}/p\widehat{\sM}_{p,\beta}  \cong   \bbF_p[x] /\beta(x)\bbF_p[x].\label{eq:DarkTea}
        \end{equation}
\end{definition}
\begin{definition}\label{def: ConditionD2}
   Fix a prime $p$.
   Then, a $\widehat{R}$-module $\sM$ is said to satisfy \emph{condition $\cD_p^{2}$} if there exists $a\in \bbZ_p$ such that \eqref{eq:DarkTea} holds for every $\beta \not\equiv x-a \bmod p$, and 
        \begin{equation}
            \widehat{\sM}_{p,x-a}/p\widehat{\sM}_{p,x-a} \cong \bbF_p[x]/(x-a)^2 \bbF_p[x]\ \ \textnormal{ and }\ \ \widehat{\sM}_{p,x-a}/(x-a) \widehat{\sM}_{p,x-a} \cong \bbF_p[x]/(x-a)\bbF_p[x].
        \end{equation}
\end{definition}

The hat in the notation $\widehat{\sM}_{p,\beta}$ serves to emphasize that the tensor product is taken over the profinite completion $\widehat{R}$, not over  $R=\bbZ[x]$ itself as in \eqref{eq: Def_NotationGpMpbeta}. 
This serves to avoid ambiguity related to $\widehat{R}$-modules also being $R$-modules but will otherwise not be significant.

In view of  \Cref{prop: ModuleRephrasingDiscriminantCondition}, the discriminant being odd can be modeled by the event that $\cD_p^1$ is satisfied for $p=2$.
Further, considering \Cref{prop: ModuleRephrasingDiscriminantConditionTwo}, additionally imposing square-freeness can be modeled by $\cD_p^1$ or $\cD_p^2$ being satisfied for all odd primes $p$. 
Thus, we find a profinite model for the setting of \Cref{conj: Discriminant}.

Recall from \eqref{eq:WeepyJar}--\eqref{eq:BoldAnt} that it holds for $\bM \sim \operatorname{SymHaar}(\widehat{R}^{n\times n})$ that the random modules $\coker(\bM-x\bI) \otimes_{\widehat{R}} R_{p,\beta}$ for varying $(p,\beta)$ are independent with the same distribution as $\coker(\bM_{p,\beta})$ for $\bM_{p,\beta}\sim \operatorname{SymHaar}(R_{p,\beta}^{n\times n})$. 
It hence suffices to study these pieces separately, which \Cref{cor: ReductionBounded} again reduces to a finite direct computation in \Cref{lem: CokMp_Trivial,lem: CokMp_Fq,lem: D2}.
We conclude in \Cref{thm: MainHaarDisc}.

\begin{lemma}\label{lem: CokMp_Trivial}
    Let $q = p^{\operatorname{deg}(\beta)}$. Then, for every $n\geq 1$, 
    \begin{align}
        \biggl\lvert \, \bbP\biggl(\frac{\coker(\bM_{p,\beta})}{p\coker(\bM_{p,\beta})}  \cong 0   \biggr) - \prod_{i\geq 0}\Bigl(1 - \frac{1}{q^{2i+1}} \Bigr)     \, \biggr\rvert \leq \frac{3}{q^n}.\label{eq:LazyKnot}
    \end{align}
\end{lemma} 
\begin{proof}
    Nakayama's lemma \cite[Proposition 2.6]{atiyah2018introduction} applied to the local ring $R_{p,\beta}/pR_{p,\beta}$ yields that the quotient $\coker(\bM_{p,\beta})/p\coker(\bM_{p,\beta})$ is trivial if and only if $\bM_{p,\beta}$ is full rank over the finite field $\bbF_q \de \bbF_{p}[x]/\beta(x)\bbF_{p}[x]$. 
    The result is hence immediate from \cite[Theorem 4.1]{fulman2015stein}; recall also \eqref{eq:NoisyHut}. 
\end{proof}

\begin{lemma}\label{lem: CokMp_Fq}
    Let $q = p^{\operatorname{deg}(\beta)}$. Then, for every $n\geq 1$, 
    \begin{align}
        \biggl\lvert \, \bbP\biggl(\frac{\coker(\bM_{p,\beta})}{p\coker(\bM_{p,\beta})}  \cong  \frac{\bbF_{p}[x]}{ \beta(x) \bbF_p[x]}   \biggr) -  \frac{1}{q} \prod_{i\geq 0}\Bigl(1 - \frac{1}{q^{2i+1}} \Bigr)    \, \biggr\rvert \leq \frac{3}{q^n} \label{eq:ZippyWok}
    \end{align}
\end{lemma}
\begin{proof}
    Denote $\bbF_q \cong \bbF_p[x]/\beta(x)\bbF_p[x]$. 
    Note that $\coker(\bM_{p,\beta})/p\coker(\bM_{p,\beta}) \cong \bbF_q$ implies that $\coker(\bM_{p,\beta})\otimes_{R_{p,\beta}} \bbF_q  \cong \bbF_q$ which is only possible if $\bM_{p,\beta}$ has rank $n-1$ over $\bbF_q$.  
    Hence, by \Cref{cor: ReductionBounded}, it holds with $k\sim \operatorname{Haar}(  pR_{p,\beta} + \beta(x)R_{p,\beta})$ that 
    \begin{equation}
        \bbP\bigl(\coker(\bM_{p,\beta})/p\coker(\bM_{p,\beta})  \cong  \bbF_q  \bigr)  = \bbP\bigl(\operatorname{rank}_q(\bM_{p,\beta})=n-1 \bigr)\bbP\bigl( {\coker(k)}/{p\coker(k)} \cong \bbF_q \bigr).\label{eq:WeavingMud}
    \end{equation}
    Decompose $k = k_1 p + k_2 \beta(x)$ for $k_1,k_2 \sim \operatorname{Haar}(R_{p,\beta})$. 
    Then, it holds that $\coker(k)/p\coker(k)\cong \bbF_q$  if and only if and only if $k_2 \not\in pR_{p,\beta} + \beta(x) R_{p,\beta}$.
    Hence, \eqref{eq:ZippyWok} follows from \eqref{eq:WeavingMud} by using \cite[Theorem 4.1]{fulman2015stein} (recall \eqref{eq:NoisyHut}) to estimate the probability of rank $n-1$ and using that $k_2$ reduces to a nonzero element in $\bbF_q$ with probability $1- 1/q$. 
    We here simplified using that $(1-1/q)/(q-1) = 1/q$. 
\end{proof}

\begin{lemma}\label{lem: D2}
    Let $q = p^{\deg(\beta)}$. 
    Then, for every $n\geq 1$,  
    \begin{align}
        \biggl\lvert \, \bbP\biggl(\frac{\coker(\bM_{p,\beta})}{p\coker(\bM_{p,\beta})}  \cong  \frac{\bbF_{p}[x]}{ \beta(x)^2 \bbF_p[x]},\,  \frac{\coker(\bM_{p,\beta})}{\beta(x)\coker(\bM_{p,\beta})} \cong \frac{\bbF_{p}[x]}{ \beta(x) \bbF_p[x]} \biggr) &-   \frac{1}{q^2}\Bigl(1 - \frac{1}{q}\Bigr) \prod_{i\geq 0}\Bigl(1 - \frac{1}{q^{2i+1}} \Bigr)    \, \biggr\rvert \nonumber\\ 
        &\leq \frac{3}{q^n}.\label{eq:DangerousNet} 
    \end{align}
\end{lemma}
\begin{proof}
    Denote $\bbF_q \de \bbF_p[x]/\beta(x)\bbF_p[x]$. 
    In particular, the event on the left-hand side of \eqref{eq:DangerousNet} implies that $\coker(\bM_{p,\beta})\otimes_{R_{p,\beta}} \bbF_q \cong \bbF_q$, which necessitates that $ \rank_q(\bM_{p,\beta}) = n-1$. 
    Hence, by the law of total probability and \Cref{cor: ReductionBounded}, it holds with $k\sim \operatorname{Haar}(  pR_{p,\beta} + \beta(x)R_{p,\beta})$ that 
    \begin{align}
        \bbP{}&{}\Bigl(\frac{\coker(\bM_{p,\beta})}{p\coker(\bM_{p,\beta})}  \cong  \frac{\bbF_{p}[x]}{ \beta(x)^2 \bbF_p[x]},\  \frac{\coker(\bM_{p,\beta})}{\beta(x)\coker(\bM_{p,\beta})} \cong \frac{\bbF_{p}[x]}{ \beta(x) \bbF_p[x]} \Bigr) \label{eq:TinyHome}\\ 
        &\quad = \bbP\Bigl( \operatorname{rank}_q(\bM_{p,\beta}) = n-1\Bigr)\bbP\Bigl(\frac{\coker(k)}{p\coker(k)}  \cong  \frac{\bbF_{p}[x]}{ \beta(x)^2 \bbF_p[x]},\  \frac{\coker(k)}{\beta(x)\coker(k)} \cong \frac{\bbF_{p}[x]}{ \beta(x) \bbF_p[x]} \Bigr).\nonumber
    \end{align}
    Decompose $k = k_1 p + k_2 \beta(x)$ for $k_1,k_2 \sim \operatorname{Haar}(R_{p,\beta})$ independent.
    Then, $\coker(k)/p\coker(k) \cong \bbF_p[x]/\beta(x)^2 \bbF_p[x]$ if and only if $k_2 \equiv c \beta(x) \bmod pR_{p,\beta} + \beta^{2}R_{p,\beta}$ for some $c\in \bbF_q \setminus \{0 \}$.
    It follows from the explicit power series representation \eqref{eq: HaarElementZpbeta} that this occurs with probability $q^{-1}(1-q^{-1})$, since we require that the first coefficient  in $k_2 = \sum_{i=0}^\infty c_i \beta^i$ vanishes modulo $p$ while the second should not.
    It further holds that $\coker(k)/\beta(x)\coker(k) \cong \bbF_p[x]/\beta(x) \bbF_p[x]$ if and only if $k_1$ has nonzero reduction in $\bbF_q$, which occurs with probability $1-q^{-1}$.
    Hence, by the independence of $k_1$ and $k_2$,  
    \begin{align}
         \bbP\Bigl(\frac{\coker(k)}{p\coker(k)}  \cong  \frac{\bbF_{p}[x]}{ \beta(x)^2 \bbF_p[x]},\  \frac{\coker(k)}{\beta(x)\coker(k)} \cong \frac{\bbF_{p}[x]}{ \beta(x) \bbF_p[x]} \Bigr) = \frac{1}{q}\Bigl(1-\frac{1}{q}\Bigr)^2. 
    \end{align}
    Using \cite[Theorem 4.1]{fulman2015stein} (recall \eqref{eq:NoisyHut}) to estimate the probability of rank $n-1$ in \eqref{eq:TinyHome} and simplifying using that $q^{-1}(1-q^{-1})^2/(q-1) = q^{-2}(1-q^{-1})$ now yields \eqref{eq:DangerousNet}.  
\end{proof}

Recall conditions $\cD_p^1$ and $\cD_p^2$ from \Cref{def: ConditionD1,def: ConditionD2}. 
Combining \Cref{lem: CokMp_Trivial,lem: CokMp_Fq,lem: D2} yields our main result concerning the satisfaction frequency of these conditions in the profinite model.
In particular, the following \Cref{thm: MainHaarDisc} justifies \Cref{conj: Discriminant}:
\begin{theorem}\label{thm: MainHaarDisc}
     Let $\bM\sim \operatorname{SymHaar}(\widehat{R}^{n\times n})$. 
     Then, for every fixed prime $p$, 
     \begin{align}
        &\lim_{n\to \infty}\bbP\Bigl(\coker(\bM - x\bI ) \textnormal{ satisfies condition }\cD_p^1 \Bigr) = \Bigl( 1 - \frac{1}{p}\Bigr) \prod_{i=1}^\infty\Bigl(1 - \frac{1}{p^{2i}}\Bigr) ,\label{eq:StormyCity} \\
        &\lim_{n\to \infty}\bbP\Bigl(\coker(\bM - x\bI ) \textnormal{ satisfies condition } \cD_p^1 \textnormal{ or }\cD_p^2 \Bigr)= \Bigl(1 - \frac{1}{p^2} \frac{3p-1}{p+1}\Bigr)\prod_{i=1}^\infty\Bigl(1 - \frac{1}{p^{2i}}\Bigr). \label{eq:FuzzyGut} 
     \end{align}
     Moreover, let $\cD^*$ denote the condition that $\cD_p^1$ or $\cD_p^{2}$ is satisfied for every odd prime $p$, and that $\cD_p^1$ is satisfied for $p=2$. 
     Then, 
     \begin{align}
        \lim_{n\to \infty}\bbP\Bigl(\coker(\bM - x\bI ) \textnormal{ satisfies condition }\cD^*    \Bigr) = \frac{6}{7} \prod_{\textnormal{primes }p}\Bigl(1 - \frac{1}{p^2} \frac{3p-1}{p+1}\Bigr)\prod_{i=1}^\infty\Bigl(1 - \frac{1}{p^{2i}}\Bigr).\label{eq:SillyKey} 
     \end{align}
\end{theorem}
\begin{proof}
    That limits may be exchanged with infinite products here follows similarly to the proof of \Cref{thm: HaarMainWalk}, so we omit these details for brevity. 
    It then follows by substituting the limiting values of \Cref{lem: CokMp_Trivial,lem: CokMp_Fq} in \Cref{def: ConditionD1} using independence that  
    \begin{align}
        \lim_{n\to \infty}\bbP\Bigl(\coker(\bM - x\bI ) \textnormal{ satisfies condition }\cD_p^1 \Bigr)  = \prod_{\beta(x)} \Bigl(1 + \frac{1}{p^{\operatorname{deg(\beta)}}} \Bigr) \prod_{i=0}^\infty \Bigl( 1 - \frac{1}{p^{(2i +1) \deg(\beta)}} \Bigr). \label{eq:CheekyBoy} 
    \end{align}
    Combining the first factor and the factor with $i = 0$ in the second product using that $(1 + p^{-\operatorname{deg}(\beta)})(1 - p^{-\operatorname{deg}(\beta)}) = 1-p^{2\operatorname{deg}(\beta)}$ and subsequently using \Cref{lem: Zeta} now yields \eqref{eq:LazyKnot}.  

    The event $\cD_p^2$ can be subdivided in $p$ mutually exclusive events depending on $a\bmod p$.   
    Hence, substituting the limiting values from \Cref{lem: CokMp_Trivial,lem: CokMp_Fq,lem: D2} in \Cref{def: ConditionD2}, 
    \begin{align}
        \lim_{n\to \infty}{}&{}\bbP\Bigl(\coker(\bM - x\bI ) \textnormal{ satisfies condition }\cD_p^2 \Bigr)\label{eq:UnripePug}\\
        &= \sum_{a=0}^{p-1} \biggl(\frac{1}{p^2} \Bigl(1 - \frac{1}{p} \Bigr)\prod_{i=0}^\infty \Bigl( 1 - \frac{1}{p^{2i +1}} \Bigr)\biggr)\prod_{\beta \not\equiv x-a}  \Bigl(1 + \frac{1}{p^{\operatorname{deg(\beta)}}} \Bigr) \prod_{i=0}^\infty \Bigl( 1 - \frac{1}{p^{(2i +1) \deg(\beta)}} \Bigr) \nonumber\\
        &=  \frac{1}{p}\frac{1 -1/p}{1+1/p}\prod_{\beta(x)}  \Bigl(1 + \frac{1}{p^{\operatorname{deg(\beta)}}} \Bigr) \prod_{i=0}^\infty \Bigl( 1 - \frac{1}{p^{(2i +1) \deg(\beta)}} \Bigr) \nonumber\\
        &=\frac{1}{p} \frac{1- 1/p}{1 + 1/p}  \lim_{n\to \infty}\bbP\Bigl(\coker(\bM - x\bI ) \textnormal{ satisfies condition }\cD_p^1 \Bigr).\nonumber
    \end{align}
    Here, the final equality used \eqref{eq:CheekyBoy}. 
    Combine \eqref{eq:StormyCity} and \eqref{eq:UnripePug} using that $\cD_p^1$ and $\cD_p^2$ are mutually exclusive, and simplify using that $(1 + p^{-1}(1 -p^{-1})/(1 + p^{-1}))\times (1- p^{-1}) = 1 - p^{-2 }(3p - 1)/(p+1)$  to find \eqref{eq:FuzzyGut}. 

    Finally, \eqref{eq:SillyKey} follows by taking the product of \eqref{eq:FuzzyGut} over all odd primes $p$, and multiplying with the value of \eqref{eq:StormyCity} at $p=2$, using that $1-p^{-1} = (6/7) (1 -p^{-2} (3p-1) /(p+1))$ for $p=2$.   
\end{proof}

\begin{remark}
    Note that while \eqref{eq:FuzzyGut} is also valid for $p=2$, the rephrasing of $p \parallel \Delta_{\bM}$ in \Cref{prop: OddSquaredDividesDiscriminant,prop: ModuleRephrasingDiscriminantConditionTwo} are only valid for odd primes.
    Indeed, $2\mid \Delta_{\phi}$ implies that $2^2 \mid \Delta_{\phi}$ for any monic $\phi \in \bbZ[x]$; see \cite[Proposition 6.7]{ash2007equality}.
    This is why \eqref{eq:BreezyGem} in \Cref{conj: Discriminant} is specific to odd primes. 
\end{remark}

\section{Conclusion}\label{sec: Conclusion}

Prior to this work, the probabilistic understanding of spectral characterization conditions was mostly limited to numerical data, especially in the presence of a symmetry constraint.
We here developed a theoretical framework that involves studying abstract-algebraic objects in analytically tractable profinite random matrix ensembles. 
In particular, this enabled the first specific conjectures on the satisfaction frequency of spectral characterization conditions. 
Further, the rigorous theory of the considered conditions is now reduced to specific but nontrivial technical challenges, such as universality results that would enable extension beyond analytically tractable ensembles.  

The developed framework has potential for application to other sufficient conditions and settings. 
One question in this regard surrounds the exceptional phenomena surrounding the prime $2$, such as the fact that different behavior for $\det(\bW)$ and $\Delta_{\bM}$ then occurs for simple graphs, or if the random vector $\zeta$ is replaced by the all-ones vector. 
Specifically, \Cref{tab: WalkOnesVec,tab: WalkIndicatorVec} suggests that the prediction in \eqref{eq: DetWalkPsquare} for the probability that $p^2$ divides the determinant is universal in terms of the vector $\zeta$, \emph{except} for the all-ones vector.
For the latter, the empirical probabilities agree with those from \Cref{tab: WalkRandVec} and Conjecture \ref{conj: WalkMatrixSquareFree} for odd primes, but a slightly different empirical probability arises when $p=2$. 

In a future work, we intend to pursue extensions of the framework of the present paper to explain such exceptional phenomena. 
This will involve richer algebraic structures and different analytically tractable random matrix models. 
\begin{table}[h!]
    \begin{center}
    \begin{tabular}{ |c|c|c|c|c|c|c|c|c| }
        \hline
        $p^2\nmid \det(\bW)$&$n=8$&$n =10$&$n=12$&$n=15$&$n = 25$&$n=50$&$n=100$& \Cref{conj: WalkMatrixSquareFree}\\
        \hline
        $p=2$&0.414&0.423&0.428&0.430&0.430&0.431&0.431&0.47336955677\ldots\\
        $p=3$&0.556&0.653&0.712&0.748&0.758&0.758&0.757&0.75752129361\ldots\\
        $p=5$&0.608&0.746&0.839&0.898&0.914&0.914&0.914&0.91393033780\ldots \\
        $p=7$&0.622&0.771&0.872&0.939&0.957&0.957&0.957&0.95674525798\ldots\\
        $p=11$&0.630&0.785&0.892&0.963&0.983&0.983&0.983&0.98279431682\ldots\\ 
        \hline
      \end{tabular}
    \end{center}
    \caption{Estimated probability that $p^2\nmid \det(\bW)$ when $\bM$ is the adjacency matrix of a random graph with loops, as in \Cref{conj: WalkMatrixSquareFree}, but $\zeta = (1,\ldots,1)^{\T}$ deterministically.
    For $p=2$, we observe a statistically significant discrepancy with the setting of \Cref{conj: WalkMatrixSquareFree} and \Cref{tab: WalkRandVec} where $\zeta$ was random. 
    These estimates used $10^6$ samples.}
    \label{tab: WalkOnesVec}
\end{table}

\begin{table}[h!]
    \begin{center}
    \begin{tabular}{ |c|c|c|c|c|c|c|c|c| }
        \hline
        $p^2\nmid \det(\bW)$&$n=8$&$n =10$&$n=12$&$n=15$&$n = 25$&$n=50$&$n=100$& \Cref{conj: WalkMatrixSquareFree}\\
        \hline
        $p=2$&0.440&0.459&0.468&0.472&0.473&0.474&0.473&0.47336955677\ldots\\
        $p=3$&0.565&0.657&0.714&0.748&0.757&0.758&0.757&0.75752129361\ldots\\
        $p=5$&0.618&0.657&0.840&0.898&0.913&0.914&0.914&0.91393033780\ldots \\
        $p=7$&0.632&0.773&0.873&0.939&0.957&0.956&0.957&0.95674525798\ldots\\
        $p=11$&0.641&0.788&0.893&0.963&0.983&0.983&0.983&0.98279431682\ldots\\ 
        \hline
      \end{tabular}
    \end{center}
    \caption{Estimated probability that $p^2\nmid \det(\bW)$ when $\bM$ is the adjacency matrix of a random graph with loops but $\zeta = (1,0,\ldots,0)^{\T}$ is the indicator vector of the first coordinate.
    The indicator vector was here arbitrarily chosen as a vector that has different structure from both the all-ones vector and a random vector, thus giving a natural test for what universality could be expected. 
    We observe that the data matches \Cref{conj: WalkMatrixSquareFree}, also for $p=2$. 
    These estimates used $10^6$ samples.}
    \label{tab: WalkIndicatorVec}
\end{table}

\subsection*{Acknowledgements}
Alexander Van Werde is funded by the Deutsche Forschungsgemeinschaft (DFG, German Research Foundation) under Germany's Excellence Strategy EXC 2044/2 –390685587, Mathematics Münster: Dynamics–Geometry–Structure. 

\subsection*{Code}
The code used to produce \Cref{tab: WalkRandVec,tab: Disc,tab: WalkOnesVec,tab: WalkIndicatorVec} can be found at \href{https://github.com/Alexander-Van-Werde/SourceCode-On-the-satisfaction-frequency-of-spectral-characterization-conditions}{\nolinkurl{https://github.com/Alexander-Van-Werde/SourceCode-On-the-satisfaction-frequency-of-spectral-characterization-conditions}}

\bibliographystyle{abbrv}

\newpage 
\appendix

\section{Proof of \texorpdfstring{\Cref{lem: ModuleMpMpbeta}}{Lemma}}\label{apx: proofModuleMpMpbeta}

\begin{proof}
    By the Cayley--Hamilton theorem applied to the finitely generated $\bbZ$-module $\sM$, there exists a non-constant monic polynomial $Q\in \bbZ[x]$ with $Q(x)\sM = 0$ \cite[Proposition 2.4]{atiyah2018introduction}. 
    Consider the factorization $Q \equiv \prod_{i} \beta_i(x)^{e_i} \bmod p$ into powers of coprime irreducible monic polynomials $\beta_i(x)\in \bbF_p[x]$.
    Hensel's lemma \cite[Theorem 4.7.2]{gouvea2020padicintro} then yields monic polynomials $Q_i \in \bbZ_p[x]$ with 
    \begin{align}
        \textstyle Q(x) = \prod_{i} Q_i(x) \qquad \textnormal{ and }\qquad Q_i(x) \equiv \beta_i(x)^{e_i} \bmod p. 
        \label{eqn:pCongruence}
    \end{align}
    The ideals $Q_i\bbZ_{p}[x]$ and $Q_j\bbZ_p[x]$ are coprime for $i\neq j$. 
    Indeed, $\bbF_p[x]/(\beta_i^{e_i}[x] \bbF_p[x] + \beta_j^{e_j}\bbF_p[x]) = 0$ by coprimality of the $\beta_i$ and hence $\bbZ_p[x]/(Q_i\bbZ_p[x] + Q_j\bbZ_p[x]) = 0$ by Nakayama's lemma over $\bbZ_p$  \cite[Proposition 2.6]{atiyah2018introduction}.  
    Hence, $\bbZ_p[x]/Q(x)\bbZ_p[x]\cong \oplus_i \bbZ_p[x]/Q_i(x)\bbZ_p[x]$ by the Chinese remainder theorem.
    Now, using that $Q(x)\sM_p = 0$ by the definition of $Q$,  
    \begin{align}
        \sM_p \cong \sM_p \otimes_{\bbZ_p[x]} \frac{\bbZ_p[x]}{Q(x) \bbZ_p[x]}  \cong \bigoplus_i \Bigl(\sM_{p}\otimes_{\bbZ_p[x]}  \frac{\bbZ_p[x]}{Q_i(x)\bbZ_p[x]} \Bigr).\label{eq:LazyYarn} 
    \end{align}
    We claim that it now remains to prove that 
    \begin{align}
        \bigl(\bbZ_p[x]/Q_i(x)\bbZ_p[x] \bigr)\otimes_{\bbZ_p[x]}R_{p,\beta_j} = \begin{cases}
            \bbZ_p[x]/Q_i(x)\bbZ_p[x] & \textnormal{ if }i=j,\\  
            0 & \textnormal{ else.}
        \end{cases} \label{eq:ZombieWisp}
    \end{align}
    Indeed, using \eqref{eq:ZombieWisp} and $\sM_{p,\beta} = \sM_p\otimes_{\bbZ_p[x]} R_{p,\beta}$ with the associativity in the tensor product in \eqref{eq:LazyYarn} then yields $\sM_p \otimes_{\bbZ_p[x]} \bbZ_p[x]/Q_i \bbZ_p[x] \cong \sM_{p,\beta}$, so that \Cref{lem: ModuleMpMpbeta} indeed follows. 

    We next prove \eqref{eq:ZombieWisp}. 
    Recall the classical fact that $I$-adic completion of a finitely generated module over a Noetherian ring is equivalent to an extension of scalars \cite[Proposition 10.13]{atiyah2018introduction}. 
    That is, if $\hat{\sN}_I$ is the $I$-adic completion of a  finitely generated $R$-module $\sN$ for an ideal $I\subseteq R$ of a Noetherian ring $R$, defined as the inverse limit of the system $\{\sN/I^n\sN:n\geq 1 \}$, then $\hat{\sN}_I \cong \sN \otimes_R \hat{R}_I$ with $\hat{R}_I$ the $I$-adic completion of the ring.   
    Denote $\sQ_i \de \bbZ_p[x]/Q_i(x)\bbZ_p[x]$ and $\mathfrak{m}_{j} = p\bbZ_{p}[x]+ \beta_j(x)\bbZ_{p}[x]$. 
    Then, 
    \begin{align}
        \sQ_i \otimes_{\bbZ_p[x]}R_{p,\beta_j} &\cong \widehat{(\sQ_i)}_{\mathfrak{\mathfrak{m}_j}} =\Bigl\{(q_n)_{n\geq 1}\in \prod_{n\geq 1}\frac{\sQ_i}{\mathfrak{m}_{j}^n \sQ_i}:\forall n,\,  q_{n+1} \equiv q_n \bmod \mathfrak{m}_{j}^n \sQ_i  \Bigr\}. \label{eq:BoldFace}
    \end{align}
    
    Let us start with the case $i\neq j$ in \eqref{eq:ZombieWisp}.
    The assumption that the $\beta_i$ have coprime reductions in $\bbF_p[x]$ implies that that $\beta_i^{e_i}\bbF_p[x] + \beta_j\bbF_p[x] = \bbF_p[x]$. 
    Thus, it follows from $Q_i \equiv \beta_i^{e_i}\bmod p$ that $\sQ_i/\mathfrak{m}_j\sQ_i \cong \bbZ_p[x]/(Q_i\bbZ_p[x] + p\bbZ_p[x] + \beta_j\bbZ_p[x]) \cong \bbF_p[x]/(\beta_i^{e_i}\bbF_p[x] + \beta_j\bbF_p[x]) \cong 0$. 
    This means that $\mathfrak{m}_j \sQ_i = \sQ_i$ and hence $\mathfrak{m}_j^n \sQ_i = \sQ_i$ for all $n\geq 1$.
    Hence, $\sQ_i \otimes_{\bbZ_p[x]}R_{p,\beta_j}\cong 0$ by \eqref{eq:BoldFace}, as desired.

    Now assume that $i=j$.
    The assumption that $Q_i \equiv \beta_i^{e_i}\bmod p$ then implies that $\beta_{i}^{e_i}\sQ_i \subseteq p \sQ_i$. 
    It follows that 
    $
    (\mathfrak{m}_{i})^{e_i} 
    \sQ_i  
    \subseteq
    p  \sQ_i \subseteq
    \mathfrak{m}_{i} \sQ_i
    $
    and hence, for all $n\geq 1$,  
    \begin{align}
        (\mathfrak{m}_{i})^{n e_i} 
    \sQ_i  
    \subseteq
    p^n  \sQ_i \subseteq
    \mathfrak{m}_{i}^n \sQ_i.\label{eq:GhostlyForce}
    \end{align}
    Inclusion of ideal filtrations implies isomorphism of the associated completions \cite[Lemma 7.14]{eisenbud2013commutative}.
    Hence, $\widehat{(\sQ_{i})}_{\mathfrak{m}_i} \cong \widehat{(\sQ_{i})}_{p\bbZ_p[x]}$. 
    Now, using that completion are the same thing as extensions of scalars \cite[Proposition 10.13]{atiyah2018introduction} (recall \eqref{eq:BoldFace}) and using that $\bbZ_p[x]$ is its own $p$-adic completion, 
    \begin{align}
        \sQ_i \otimes_{\bbZ_p[x]}R_{p,\beta_i} \cong \widehat{(\sQ_{i})}_{\mathfrak{m}_i} \cong \widehat{(\sQ_{i})}_{p\bbZ_p[x]} \cong  \sQ_{i}\otimes_{\bbZ_p[x]}\bbZ_p[x]\cong  \sQ_i. 
    \end{align}
    This yields the case $i=j$ in \eqref{eq:ZombieWisp} and hence concludes the proof. 
\end{proof}

\section{Recovering conjectures without symmetry constraint}\label{sec: ConjWithoutSymmetry}
Recall we claimed in \Cref{rem: AlternativeConjectureWithoutSymmetry} that one can also use the profinite Haar method to recover \cite[Conjecture 1.4]{van2025cokernel} concerning walk matrices in a setting without symmetry. 
The conjecture is there stated in a setting with sparse matrices and specifically with odd primes and the all-ones vector, but those complications should not affect the universality class in the setting without symmetries. 
(A universality statement to this effect is given in \cite[Theorem 1.3]{van2025cokernel}.) 
Hence, the key content that we wish to give an alternative justification for is as follows:

\begin{conjecture}\label{conj: WithoutSymmetry} 
    Let $\bM$ be uniformly distributed on $\{0,1 \}^{n\times n}$ and consider an independent uniform random vector $\zeta\in \{0,1 \}^{n\times n}$. 
    Then, $\bW = [\zeta, \bM \zeta, \ldots, \bM^{n-1}\zeta]$ satisfies that for every prime $p$,
    \begin{align}
        \lim_{n\to \infty} \bbP\Bigl(p^2 \nmid \det(\bW) \Bigr) &= \Bigl(1 + \frac{1}{p} \Bigr) \prod_{i=1}^\infty \Bigl(1 -  \frac{1}{p^{i}}\Bigr)
    \end{align}
\end{conjecture}
The difference with the setting of \Cref{conj: WalkMatrixSquareFree} is that the symmetry constraint is now removed.
This makes the sufficient condition \Cref{thm: HaarMainWalk} non-applicable, but the walk matrix itself is naturally still a well-defined object and the rephrasing in \Cref{prop: ModuleRephrasingWalkDet} is still applicable. 

Similar to \Cref{sec: ModelDefinition,sec: ConditionWalk}, we model \Cref{conj: WithoutSymmetry} by studying the satisfaction frequency of condition $\cW_p$ for $\coker(\bM -x\bI, \zeta)$ when $\bM \sim \operatorname{Haar}(\widehat{R}^{n\times n})$ and $\zeta \sim \operatorname{Haar}(\widehat{R}^{n})$.   
By translation invariance and independence, this again reduces to the study of the pieces over the local ring $R_{p,\beta}$.

\subsection{Reduction result}
Let $(\sO,\mathfrak{m})$ be a compact Noetherian local ring and recall that $\operatorname{rank}_q(\cdot)$ refers to the rank of a matrix over $\bbF_q \cong \sO/\mathfrak{m}$. 
It follows from \cite[Theorem 1.1]{fulman2015stein} that a uniform random rectangular matrix $\bU \in \bbF_q^{n\times (n +m)}$ with $m\geq 0$ satisfies that for every $k\leq n$,
\begin{align}
    \biggl\lvert \bbP\bigl(\rank_q(\bU) = n-k \bigr) - \frac{1}{q^{k (m+k)}} \frac{\prod_{i=k+1}^\infty(1-1/q^i) }{\prod_{i=1}^{k+m} (1-1/q^i)} \biggr\rvert \leq \frac{3}{q^n}.\label{eq:GoodOtter} 
\end{align}
The cokernel conditional on the rank is described by the following variant on \Cref{prop: Reduction}:

\begin{proposition}\label{prop: ReductionWithoutSym}
    Suppose that $\bM \sim \operatorname{Haar}(\sO^{n\times n})$. 
    Fix some $k\leq n$ and let $\bK  \sim \operatorname{Haar}(\mathfrak{m}^{k\times k})$. 
    Then, we have an equality in distribution of random $\sO$-modules: 
    \begin{align}
        \bbP\Bigl(\coker\bigl(\bM \bigr)  \in * \mid \rank_{q}\bigl(\bM \bigr) = n-k\Bigr) = \bbP\Bigl(\coker\bigl(\bK\bigr) \in * \Bigr).\label{eq:OldVan}
    \end{align}
\end{proposition} 
\begin{proof}
    This follows analogously to \Cref{prop: Reduction}. 
    Indeed, similar to \Cref{lem: InvarianceSymHaar} the $\operatorname{Haar}(\mathfrak{a}^{n\times n})$ distribution for any ideal $\mathfrak{a}\subseteq \sO$ is invariant under the map $\bM \mapsto \bG_1 \bM \bG_2$ for any fixed $\bG_1,\bG_2 \in \operatorname{GL}_{n}(\sO)$.
    Further, similar to \Cref{lem: BlockDiagonalDecomposition} a matrix $\bM \in \sO^{n\times n}$ has rank $n-k$ if and only if there exist invertible $\bG_1,\bG_2 \in \operatorname{GL}_n(\sO)$ with $\bG_1 \bM \bG_2 = \operatorname{diag}(\bQ, \bK)$ for some $\bK \in \mathfrak{m}^{k\times k}$ and invertible $\bQ \in \operatorname{GL}_{n-k}(\sO)$.  
    Using the aforementioned invariance similarly to the arguments preceding \eqref{eq:SafeYawn} allows us to assume without loss of generality that $\bK \sim \operatorname{Haar}(\mathfrak{m}^{k\times k})$, which yields \eqref{eq:OldVan}. 
\end{proof}

\begin{corollary}\label{cor: ReductionBoundedWithoutSym}
    Suppose that $\bM_{p,\beta} \sim \operatorname{Haar}(R_{p,\beta}^{n\times n})$. 
    Let $\rank_q(\cdot)$ refer to the rank over $\bbF_q \cong \bbZ[x]/(p\bbZ[x] + \beta(x)\bbZ[x])$. 
    Then, for every $k\leq n$,  as random $R_{p,\beta}$ modules, 
    \begin{align}
        \bbP\Bigl(\coker\bigl(\bM_{p,\beta},\bigr) \in  *  \mid \rank_{q}\bigl(\bM_{p,\beta}\bigr) = n-k\Bigr) = \bbP\Bigl(\coker\bigl(\bK  \bigr) \in  * \Bigr),
    \end{align}
    where $\bK$ is a $k\times k$ matrix with entries Haar distributed on $pR_{p,\beta} + \beta(x) R_{p,\beta}$. 
    That is, 
    \begin{align}
        \bK  \sim  p\operatorname{Haar}(R_{p,\beta}^{k\times k})  + \beta(x) \operatorname{Haar}(R_{p,\beta}^{k\times k}).   
    \end{align}
\end{corollary}    
\begin{proof}
    This is immediate from \Cref{prop: ReductionWithoutSym} with $\sO = R_{p,\beta}$. 
\end{proof}

\subsection{Satisfaction frequency of condition \texorpdfstring{$\cW_p$}{Wp}}
Let $\bM_{p,\beta} \sim \operatorname{Haar}(R^{n\times n}_{p,\beta})$ and $\zeta_{p,\beta} \sim \operatorname{Haar}(R^n_{p,\beta})$. 
The following \Cref{lem: TrivialCoker_pbetaWithoutSym,lem: 1D-cokernelWithoutSym} replace the role of \Cref{lem: TrivialCoker_pbeta,lem: 1D-cokernel}.

\begin{lemma}\label{lem: TrivialCoker_pbetaWithoutSym}
    Let $q \de p^{\operatorname{deg}(\beta)}$. 
    Then, for every $n\geq 1$, 
    \begin{align}
       \Bigl\lvert\, \bbP\bigl( \coker(\bM_{p,\beta}, \zeta_{p,\beta}) \cong 0\bigr) -  \prod_{i=2}^\infty\Bigl(1- \frac{1}{q^i}\Bigr) \,  \Bigr\rvert \leq  \frac{3}{q^n}. \label{eq:TediousPop2}
    \end{align}
\end{lemma}
\begin{proof}
    Let $\bbF_q \de \bbF_{p}[x]/\beta(x) \bbF_p[x]$. 
    Then, 
    \begin{align}
        \bbP\bigl( \coker(\bM_{p,\beta}, \zeta_{p,\beta}) \cong 0\bigr) = \bbP\bigl(\operatorname{rank}_q(\bM_{p,\beta}, \zeta_{p,\beta}) = n \bigr).  
    \end{align}
    Applying \eqref{eq:GoodOtter} with $\bU = [\bM_{p,\beta}, \zeta_{p,\beta}] \bmod p,\beta$ by taking $m=1$ and $k=0$ yields \eqref{eq:TediousPop2}.  
\end{proof}

\begin{lemma}\label{lem: 1D-cokernelWithoutSym}
    Let $q \de p^{\deg(\beta)}$.
    Then, every $n\geq 2$, 
    \begin{align}
          \biggl\lvert\, \bbP\biggl(\coker\bigl(\bM_{p,\beta},\zeta_{p,\beta}\bigr)    \cong \frac{\bbF_{p}[x]}{\beta(x)\bbF_{p}[x]}\biggr) - \frac{1}{q^2} \prod_{i=2}^\infty\Bigl(1-\frac{1}{q^i}\Bigr)   \,  \biggr\rvert \leq \frac{6}{q^n}. \label{eq:SafeKee2}
    \end{align}
\end{lemma}
\begin{proof}
    
    Let $\bbF_q \de \bbF_{p}[x]/\beta(x)\bbF_{p}[x]$. 
    By the law of total probability, 
    \begin{align}
        \bbP\bigl(\coker\bigl(\bM_{p,\beta},{}&{}\zeta_{p,\beta}\bigr)    \cong  \bbF_q\bigr) \label{eq:BlueKey2}\\ 
        & = \bbP\bigl(\coker\bigl(\bM_{p,\beta},\zeta_{p,\beta}\bigr)    \cong\bbF_q\ \Big\vert \   \rank_q(\bM_{p,\beta}) = n-1\bigr) \bbP\bigl( \rank_q(\bM_{p,\beta}) = n-1\bigr) \nonumber\\ 
        & \ \ +  \bbP\bigl(\coker\bigl(\bM_{p,\beta},\zeta_{p,\beta}\bigr)    \cong\bbF_q \ \Big\vert \   \rank_q(\bM_{p,\beta}) = n-2\bigr) \bbP\bigl( \rank_q(\bM_{p,\beta}) = n-2\bigr).\nonumber
    \end{align}
    
   The probability conditional on rank $n-1$ can be estimated as in \eqref{eq:FunBox}--\eqref{eq:DizzyCrow} by using \Cref{cor: ReductionBoundedWithoutSym} instead of \Cref{cor: ReductionBounded}. 
   The arguments are entirely unchanged because symmetry constraints had no effect in the one-dimensional computations involved. 
   In particular, we again find that 
    \begin{align}
         \bbP\bigl(\coker\bigl(\bM_{p,\beta(x)},\zeta_{p,\beta}\bigr)    \cong \bbF_q  \mid   \rank_q(\bM_{p,\beta}) = n-1\bigr) = q^{-1} \bigl(1 - q^{-2}\bigr) \bigl(1-q^{-1} \bigr).\label{eq:DizzyCrow2} 
    \end{align}

   Conditional on rank $n-2$,  \Cref{cor: ReductionBoundedWithoutSym} yields that $\coker(\bM_{p,\beta})$ then has the same distribution as $\coker(\bK)$ with $\bK$ a $2\times 2$ matrix with entries Haar distributed on $ pR_{p,\beta} + \beta(x)R_{p,\beta}$. 
    Then, $\coker(\bM_{p,\beta},\zeta_{p,\beta})$ has the same distribution as $\coker(\bK, Z)$ with $Z\sim \operatorname{Haar}(R_{p,\beta}^2)$. 
    Here, using that $Z$ reduces to zero in $\bbF_q^2$ with probability $1/q^{2}$ as in \eqref{eq:PaleTea},  
    \begin{align}
         \bbP\bigl(\coker\bigl(\bM_{p,\beta},\zeta_{p,\beta}{}&{}\bigr)    \cong\bbF_q  \mid    \rank_q(\bM_{p,\beta}) = n-2\bigr)\label{eq:PaleTea2} \\ 
         &= \bigl( 1 - q^{-2}\bigr) \bbP\bigl( \coker(\bK , Z) \cong \bbF_q  \mid Z\not\equiv 0 \bmod p R_{p,\beta} + \beta(x) R_{p,\beta}  \bigr).\nonumber 
    \end{align}
   That $Z$ has nonzero reduction is equivalent to the existence of  $\bG \in \operatorname{GL}_2(R_{p,\beta})$ with $\bG Z = (1,0)^{\T}$.     
   The invariance of $\operatorname{Haar}(\mathfrak{m}_{p,\beta}^{2\times 2})$ here implies that $\tilde{\bK}=\bQ \bK$ has the same distribution as $\bK$. 
    Further,  
    \begin{equation}
        \coker(\bK, Z) \cong \coker(\bQ \bK , \bQ Z)  
        = \coker\begin{pmatrix}
            \tilde{\bK}_{1,1} & \tilde{\bK}_{1,2} & 1 \\ 
            \tilde{\bK}_{2,1} & \tilde{\bK}_{2,2} & 0
        \end{pmatrix} 
        \cong  \coker(\tilde{\bK}_{2,1},\, \tilde{\bK}_{2,2}).  
    \end{equation}
    From here on, the arguments are again identical to those for \Cref{lem: 1D-cokernel}. 
    We thus again find that 
    \begin{align}
        \bbP\bigl(\coker\bigl(\bM_{p,\beta}, \zeta_{p,\beta}\bigr)    \cong  \bbF_q\bigr) ={}&{}  q^{-1} \bigl(1 - q^{-2}\bigr) \bigl(1-q^{-1} \bigr) \bbP\bigl( \rank_q(\bM_{p,\beta}) = n-1\bigr) \label{eq:OccultDew2}  \\ 
        & \quad +  \bigl( 1 - q^{-2} \bigr)^2 \bigl( 1 - q^{-1}\bigr)\bbP\bigl( \rank_q(\bM_{p,\beta}) = n-2\bigr). \nonumber 
    \end{align}
    The result \eqref{eq:SafeKee2} now follows from \eqref{eq:GoodOtter} with $m=0$ and $k\in \{1,2\}$ by simplifying prefactors, since
    \begin{equation}
         \frac{(1-1/q^{2}) (1-1/q)}{q} \frac{\prod_{i=2}^\infty(1-1/q^i) }{q(1-1/q)} +\frac{(1-1/q)(1-1/q^{2})^2}{q^{4}} \frac{\prod_{i=3}^\infty(1-1/q^i) }{\prod_{i=1}^{2} (1-1/q^i)}  = \frac{1}{q^2} \prod_{i=2}^\infty\Bigl(1-\frac{1}{q^i}\Bigr).  
    \end{equation}
\end{proof}

Recall \Cref{def: ConditionW} concerning condition $\cW_p$. 
Combining \Cref{lem: TrivialCoker_pbetaWithoutSym,lem: 1D-cokernelWithoutSym} yields the satisfaction frequency of this condition in the absence of a symmetry constraint. 
In particular, this provides an alternative justification for \Cref{conj: WithoutSymmetry}. 

\begin{theorem}\label{thm: HaarWalk2}
    Let $\bM\sim \operatorname{Haar}(\widehat{R}^{n\times n})$ and $\zeta \sim \operatorname{Haar}(\widehat{R}^n)$. 
    Then, for every fixed prime $p$, 
    \begin{align}
       \lim_{n\to \infty}\bbP\Bigl(\coker(\bM-x\bI, \zeta) \textnormal{ satisfies condition }\cW_p \Bigr) = \Bigl(1 + \frac{1}{p}\Bigr) \prod_{i=1}^\infty \Bigl(1 - \frac{1}{p^{i}}\Bigr).\label{eq:SlowWisp2} 
    \end{align}
\end{theorem}
\begin{proof}
    That limits may be exchanged with infinite products here follows similarly to the proof of \Cref{thm: HaarMainWalk}, so we omit these details for brevity.
    Using the translation invariance of the Haar distribution, studying $\coker(\bM-x\bI, \zeta)$ may be reduced to modules of the form $\coker(\bM_{p,\beta}, \zeta_{p,\beta})$.  
    Hence, using that (W1) and (W2) are mutually exclusive in \Cref{def: ConditionW} and subdividing the case (W2) by the $p$ possibilities for $a$ and using independence, it follows from \Cref{lem: TrivialCoker_pbetaWithoutSym,lem: 1D-cokernelWithoutSym} that 
    \begin{align}
        \lim_{n\to \infty}{}&{}\bbP\Bigl(\coker(\bM - x\bI, \zeta ) \textnormal{ satisfies condition }\cW_p \Bigr)  \\ 
        &=\prod_{\beta(x)}  \prod_{i=2}^\infty \Bigl( 1 - \frac{1}{p^{i \deg(\beta)}} \Bigr) + \sum_{a=0}^{p-1}\frac{1}{p^2} \prod_{\beta(x)}  \prod_{i=2}^\infty \Bigl( 1 - \frac{1}{p^{i \deg(\beta)}} \Bigr) = \Bigl( 1 + \frac{1}{p} \Bigr) \prod_{\beta(x)}  \prod_{i=2}^\infty \Bigl( 1 - \frac{1}{p^{i \deg(\beta)}} \Bigr).\nonumber 
    \end{align}
    Use \Cref{lem: Zeta} to conclude. 
\end{proof}

\end{document}